\documentclass[10pt]{article}
\usepackage{algorithm}
\usepackage{amsmath}
\usepackage{amsfonts}
\usepackage{amssymb}
\usepackage{amsthm}
\usepackage{mathtools}
\usepackage{authblk}
\usepackage{breqn}
\usepackage{epsf}
\usepackage[english]{babel}
\usepackage{graphicx}
\usepackage{glossaries}
\usepackage{lipsum}
\usepackage{psfrag}
\usepackage{subfig}
\usepackage[hidelinks]{hyperref}
\usepackage{tikz}
\usepackage{url}
\usepackage{epstopdf}
\usepackage{tabularx}
\usepackage{tikz}
\usepackage{setspace}
\usetikzlibrary{arrows}

\usetikzlibrary{decorations.markings}

\newtheorem{thm}{Theorem}[section]
\newtheorem{coro}[thm]{Corollary}
\newtheorem{lem}[thm]{Lemma}

\newtheorem{rem}[thm]{Remark}
\newtheorem{exa}[thm]{Example}
\newtheorem{assu}[thm]{Assumption}

\theoremstyle{definition}
\newtheorem{defi}[thm]{Definition}

\newcommand{\field}[1]{\mathbb{#1}}
\newcommand{\R}{\field{R}}
\newcommand{\N}{\mathbb{N}}
\newcommand{\vu}{\mathbf{u}}
\newcommand{\vn}{\mathbf{n}}
\newcommand{\vy}{\mathbf{y}}
\newcommand{\vv}{\mathbf{v}}
\newcommand{\vw}{\mathbf{w}}
\newcommand{\vz}{\mathbf{z}}

\newcommand{\normal}{n}
\newcommand{\id}{\operatorname{id}}
\newcommand{\argmin}{\mathrm{\,arg\,min}}
\newcommand{\abs}[1]{\left| #1 \right|}
\newcommand{\norm}[1]{\left\| #1 \right\|}
\newcommand{\inner}[2]{\left< #1, #2 \right>}
\newcommand{\set}[1]{\left\{ #1 \right\}}

\newcommand{\hone}[1]{{\mathcal{H}}^1(#1)}
\newcommand{\htm}[1]{{\widetilde{\mathcal{H}}_{\mathcal{T}}^{1}(#1)}}
\newcommand{\hn}[1]{\mathcal{H}_{\mathcal{N}}^1(#1)}
\newcommand{\leb}[1]{L^{#1}(\mathcal{M})}
\newcommand{\lebtwo}[1]{L^{#1}(\mathcal{M}_2)}
\newcommand{\xs}[1]{{\mathcal{U}}_{#1}}
\newcommand{\vx}[1]{{\mathbf{x}}_{#1}}
\newcommand{\Hone}[1]{{H}^1(#1)}

\newcommand{\parametrization}{m}

\newcommand{\surface}{{\mathcal{M}}}
\newcommand{\tv}{{\widetilde{\vu}}}
\newcommand{\udag}{\vu^\dagger}

\newcommand{\trace}{\text{Tr}}
\newcommand{\projection}[1]{\mathcal{P}_{#1}}
\newcommand{\gradientM}{{\nabla}_{\mathcal{M}}}
\newcommand{\Ptm}{\projection{\tau}}
\newcommand{\Pn}{\projection{n}}
\newcommand{\TB}{\mathcal{TM}}

\numberwithin{equation}{section}

\title{Convergence of Tikhonov regularization for solving ill-posed operator equations with solutions defined on surfaces}

\begin{document}
\font\myfont=cmr12 at 10pt
\author{\myfont Guozhi Dong\footnote{Computational Science Center, University of Vienna, 
Oskar-Morgenstern Platz 1, 1090 Vienna, Austria (\url{guozhi.dong@univie.ac.at})}, 
Bert J\"uttler\footnote{Institute of Applied Geometry, Johannes Kepler University of Linz, 
Altenberger Str. 69, A-4040 Linz, Austria (\url{bert.juettler@jku.at})}, 
Otmar Scherzer\footnote{Computational Science Center, University of Vienna, Oskar-Morgenstern Platz 1, 1090 Vienna, and Johann Radon Institute for Computational and Applied Mathematics (RICAM), Austrian Academy of Sciences, Altenberger Str. 69, A-4040 Linz, Austria (\url{otmar.scherzer@univie.ac.at})},
Thomas Takacs\footnote{Institute of Applied Geometry, Johannes Kepler University of Linz, 
Altenberger Str. 69, A-4040 Linz, Austria (\url{thomas.takacs@jku.at})}}
\date{}
\maketitle

\begin{abstract}
We study Tikhonov regularization for solving ill--posed operator equations where the 
solutions are functions defined on surfaces. One contribution of this paper is an 
error analysis of Tikhonov regularization which takes into account perturbations of 
the surfaces, in particular when the surfaces are approximated by spline surfaces.
Another contribution is that we highlight the analysis of regularization for functions with range in vector bundles over surfaces.  
We also present some practical applications, such as an inverse problem of gravimetry 
and an imaging problem for denoising vector fields on surfaces, and show the numerical verification.
\end{abstract}

\section{Introduction}                                                                                                                                                                                                                                                                                                                                                    We are interested in solving linear \emph{inverse problems}
\begin{equation}\label{eq:linear}
 F \vu = \vy,
\end{equation}
where $F : \xs{1}\to \xs{2}$ is a bounded operator between Hilbert spaces of 
functions defined on \emph{surfaces}.
Note that the functions in spaces $\xs{1}$ and $\xs{2}$ might be defined on different surfaces $\surface_1$ and $\surface_2$ respectively.
Especially we pay attention to applications (see also Section \ref{sec:vectorfields} and \ref{sec:numerics})
where $\xs{1}$ denotes a space of functions from a surface $\surface_1$ into the 
\emph{vector bundle} or into the \emph{tangent vector bundle}.

We assume that solving \eqref{eq:linear} is ill--posed and requires some regularization 
to approximate the solution in a stable way. The regularization method of choice for solving \eqref{eq:linear} 
is Tikhonov regularization, which consists in approximating a solution of \eqref{eq:linear} by a minimizer 
of the functional 
\begin{equation}
 \label{eq:Tik}
 \mathcal{T}_{\alpha,\vy^\delta}(\vu) := \norm{F\vu-\vy^\delta}_{\xs{2}}^2 + \alpha \mathcal{R}(\vu)
\end{equation}
over $\xs{1}$. 
Here $\vy^\delta$ denotes some approximation of the exact data $\vy$, from which we assume that  
\begin{equation}
 \label{eq:delta}
 \norm{\vy - \vy^\delta}_{\xs{2}} \leq \delta. 
\end{equation}
Moreover, $\alpha>0$ is a regularization parameter, and $\mathcal{R}$ is an appropriate regularization functional. 
In contrast to previous work \cite{Gro84,Mor84,TikGonSteYag95,EngHanNeu96,Tan97,TikLeoYag98,SchGraGroHalLen09,SchKalHofKaz12} 
we are emphasizing on Hilbert spaces of functions defined on \emph{surfaces}.  

The analysis of Tikhonov regularization has advanced from a theory for solving ill--posed operator equations in a 
Hilbert space setting, both in a finite and an infinite dimensional setting (see for instance \cite{TikArs77,Gro84,Mor84,Tan97,Han10}), 
to more complex situations, when $F$ is nonlinear \cite{BakGon94,TikGonSteYag95,EngHanNeu96,TikLeoYag98} or with sophisticated 
regularization \cite{BurOsh04,ResSch06,HofKalPoeSch07,SchGraGroHalLen09,PoeResSch10,SchKalHofKaz12,MueSil12}, and in statistical setting 
\cite{KaiSom05}. 
The analysis of finite dimensional approximations of regularizers in infinite dimensional spaces (see for instance 
\cite{NeuSch90,PoeResSch10,KirPoeResSch15}) is highly relevant for this work. The results of \cite{PoeResSch10} are 
already quite general, but do \emph{not} cover approximations of solutions of \eqref{eq:linear} 
by functions defined on approximating surfaces. Note that for functions defined on subsets of an Euclidean space, the 
finite dimensional functions approximating infinite dimensional functions are also defined on the Euclidean space, and this is 
not the case anymore if the underlying manifold is approximated. In this paper the differential geometrical concept of a 
pullback is used to compare functions defined on different surfaces.

The second issue of this paper is, that the solution of \eqref{eq:linear} may be a function with range 
in a \emph{vector bundle}. Consequently, when the surface is approximated also the vector bundle, or in other words, the 
range of the function is varying by approximation.
To resolve this issue we consider vector fields represented with \emph{ambient} bases.
It is then presented in the paper that the analysis of regularization for vector fields on surfaces can be incorporated into a general framework developed for vector valued functions on surfaces.

We consider several applications and perform numerical tests. One example concerns reconstructing the \emph{magnetization} 
from measurements of the magnetic potential \cite{Bla96,GubIveMasWin11,Ger16}. Another numerical test example concerns 
\emph{vector field denoising}.
Our analysis on regularizing tangent vector fields in ambient spaces provides a different point of view to 
existing work on tangent vector field regularization (see for instance \cite{LefBai08,KirLanSch13,KirLanSch15}).

The paper is organized as follows:
In Section \ref{sec:theory} we first review standard regularization results in an infinite dimensional setting, 
and then perform a convergence analysis of Tikhonov regularization taking into 
account approximated surfaces, which cannot be handled with the existing theory. 
In Section \ref{sec:vectorfields}, we introduce proper spaces for functions with range in vector bundles,
and discuss applications of the regularization theory for recovering vector fields in ambient spaces.
We make an additional discussion on recovering tangent vector fields there. 
Finally, in Section \ref{sec:numerics}, we present some examples and verify the theoretical results 
numerically. 
The underlying geometric concepts, 
including all the basic notations, are summarized in the Appendix \ref{appendix:background}.

\section{Regularization for functions defined on surfaces}
\label{sec:theory}

Below we first review results from the regularization literature (see \cite{Gro84,EngHanNeu96,SchGraGroHalLen09}), which provide 
well-posedness, convergence and stability of Tikhonov regularization. 
For this purpose we summarize the basic assumptions needed throughout this paper.
\begin{assu}[Chapter 3, \cite{SchGraGroHalLen09}]
\label{assu:ori}
In the context of Tikhonov regularization \eqref{eq:Tik} we make the following assumptions:
\begin{enumerate}
\item Let $\xs{1}$, $\xs{2}$ be two Hilbert spaces. 
\item $F :\mathcal{D}(F) \subseteq \xs{1}\to \xs{2}$ is a bounded linear operator with 
      domain $\mathcal{D}(F)$, which is weakly sequentially closed.
\item $\mathcal{R}: \xs{1}\to [0,\infty]$ is a proper, convex, and sequentially weakly lower-semi continuous 
      functional satisfying 
      $\mathcal{D} := \mathcal{D}(\mathcal{R})\cap\mathcal{D}(F) \neq \emptyset.$
\item For every $\alpha,\theta >0$ the sets 
      $M_{\alpha,\vy}(\theta):=\set{\vu \in \xs{1}: \mathcal{T}_{\alpha,\vy}(\vu)\leq \theta}$ 
      are weakly sequentially compact.
\end{enumerate}
\end{assu} 

\subsection{Standard regularization theory} 
We review some regularization results from the literature:
\begin{thm}[Chapter 3, \cite{SchGraGroHalLen09}]
\label{thm:directreg}
 Let $F$, $\mathcal{R}$, $\mathcal{D} $, $\xs{1}$ and $\xs{2}$ satisfy Assumption \ref{assu:ori}, and 
 $\vy$ and $\vy^\delta$ satisfy \eqref{eq:delta}. 
 \begin{itemize}
 \item Assume that $\alpha > 0$. Then there exists a minimizer of $\mathcal{T}_{\alpha,\vy^\delta}$.
       If $(\vy_k)_{k\in \N}$ is a sequence converging to $\vy$ in $\xs{2}$, then every sequence 
       \[\left(\vu_k := \argmin \set{\mathcal{T}_{\alpha,\vy_k}(\vu): 
        \vu \in \mathcal{D}}\right)_{k\in \N}\]
        has a weakly convergent subsequence in $\xs{1}$, and the limit is a minimizer of 
        $\mathcal{T}_{\alpha,\vy}$.
        
 \item If there exists $\vu_0\in \mathcal{D} $ such that
       \begin{equation} \label{eq:usol}
          F \vu_0=\vy.
       \end{equation}
       Then there exists an $\mathcal{R}$ minimizing solution $\vu^\dagger$. That is 
       \begin{equation*}
          \mathcal{R}(\vu^{\dagger}) = \min \set{\mathcal{R}(\vu) : \vu \in \mathcal{D} , F \vu=\vy}.
       \end{equation*}
 \item Assume that $\alpha = \alpha(\delta):(0,\infty) \to (0,\infty)$ satisfies
       \begin{equation} \label{eq:parameter_choice}
          \alpha \to 0, \; \frac{\delta^2}{\alpha} \to 0, \text{ as } \delta \to 0.
       \end{equation}
\end{itemize}
       Let $(\vy_k)_{k \in \N}$ satisfy $\norm{\vy-\vy_k}_{\xs{2}} \leq \delta_k$ and set 
       $\alpha_k = \alpha(\delta_k)$. Then every sequence 
       $(\vu_k = \text{argmin} \set{\mathcal{T}_{\alpha_k,\vy_k}(\vu) :\vu \in \mathcal{D}})$ 
       has a weakly convergent subsequence, and the limit is an $\mathcal{R}$-minimizing solution. 
\begin{itemize}
 \item Moreover, if in addition $\udag$ satisfies a \emph{source condition},
       which assumes that there exists an element $W \in \xs{2}$, such that 
       \begin{equation} \label{eq:sourcecondition}
        F^* W\in   \partial \mathcal{R}(\udag).
       \end{equation}
       Here $\partial \mathcal{R}(\udag)$ denotes the subgradient of $\mathcal{R}$ at $\udag$,
       and $F^*: \xs{2} \to \xs{1}$ is the adjoint operator of $F$.
\end{itemize}
       Then, the $\mathcal{R}$-minimizing solution satisfies
       \begin{equation*}
        B_{\mathcal{R}}(\vu_k,\udag)=
        \mathcal{O}\left(\max\set{\frac{\delta_k^2}{\alpha_k},\alpha_k }\right),
       \end{equation*}
        where $B_{\mathcal{R}}(\cdot,\cdot)$ denotes the Bregman distance with respect to the convex functional $\mathcal{R}$, 
        see for instance \cite{BurOsh04}.
 \end{thm}
 
\subsection{Convergence analysis taking into account surface perturbations}
\label{sec:analysis}
In the following we study Tikhonov regularization, consisting in minimizing the 
functional $\mathcal{T}_{\alpha,\vy^\delta}$ in \eqref{eq:Tik}, where 
$F:\xs{1} \to \xs{2}$ is an operator between function spaces 
\[\xs{i}:=\mathcal{X}_i(\surface_i):=\set{\vx{i}:\surface_i \rightarrow \R^{r_i}},\text{ for } r_i \in \N^+,\; \text{ and } i=1,2,\]
of functions defined on surfaces $\surface_{i}$, $i=1,2$, respectively. 
We consider families of surfaces $\set{\surface_{i,\sigma}}_{\sigma>0}$ approximating the surfaces 
$\surface_{i,0}:=\surface_i$, $i=1,2$. 
In what follows we use the following abbreviations:
\begin{equation}
\label{eq:simplified_norm} 
\xs{i,\sigma} := \mathcal{X}_i(\surface_{i,\sigma})\,,\quad 
\norm{\cdot}_{i,\sigma} := \norm{\cdot}_{\xs{i,\sigma}}\,,\quad i=1,2,\; \sigma > 0\,,
\end{equation}
where we delete the superscript $0$ in case $\sigma=0$. That is $\norm{\cdot}_1=\norm{\cdot}_{1,0}$.
We need a few assumptions on the approximating surfaces.
\begin{assu}
\label{assu:Surface}
Let $\set{\surface_{i,\sigma}}_{\sigma>0}$ be a sequence of surfaces approximating the surfaces $\surface_i$ for $i=1,2$,
then we assume the following properties and estimates hold:
\begin{itemize}
\item For every $i=1,2$ and every $\sigma \geq 0$, every surface $\surface_{i,\sigma}$ can be parametrized by a patch
      \begin{equation*}
       \parametrization_{i,\sigma}: \Omega_i \subset \R^{d_i} \to \surface_{i,\sigma} \subset \R^{d_i+1}\,,\quad \sigma \geq 0,
      \end{equation*}
      with the same parameter domain $\Omega_i$, and $\parametrization_{i,\sigma}$ is a bijection.
\item The operators
      \begin{equation}
 \label{eq:transform}
 \begin{aligned}
  T_{i,\sigma}: \xs{i} &\to \xs{i,\sigma}\\
     \vx{i} &\mapsto \vx{i} \circ \parametrization_i \circ (\parametrization_{i,\sigma})^{-1}
\end{aligned}
\end{equation}
and the inverse
\begin{equation}
 \label{eq:itransform}
 \begin{aligned}
  (T_{i,\sigma})^{-1}: \xs{i,\sigma} &\to \xs{i}\\
      \vx{i,}{_\sigma} &\mapsto   \vx{i,}{_\sigma} \circ \parametrization_{i,\sigma} \circ \parametrization_i^{-1}
\end{aligned}
\end{equation}
are uniformly bounded. Here by bounded we mean that there exists a real constant $C$ such that 
      \begin{equation}
      \label{eq:C}
       \norm{T_{i,\sigma}\vx{i}}_{i,\sigma}\leq C \norm{\vx{i}}_i\;\text{  and  } 
       \norm{(T_{i,\sigma})^{-1}\vx{i,}{_\sigma}}_i \leq C \norm{\vx{i,}{_\sigma}}_{i,\sigma}.
      \end{equation}
 \item Let the family of operators $F_\sigma: \xs{1,}{_\sigma} \to \xs{2,}{_\sigma}$, $\sigma > 0$ 
       and the operator $F: \xs{1}  \to \xs{2}$ satisfy
       \begin{equation} \label{eq:rhosigma}
       \rho(\sigma) := \norm{(T_{2,\sigma})^{-1}F_\sigma T_{1,\sigma} - F} \to 0 \text{ for } \sigma \to 0 . 
       \end{equation} 
 \item Let $\mathcal{R}_\sigma:\xs{1,}{_\sigma} \rightarrow [0,\infty]$ with $\sigma >0$ be the family of regularization functionals, 
 	   then there exists $C_M\in \R$ such that
       \begin{equation}
       \label{eq:rsigma}
       \abs{\mathcal{R}_\sigma(T_{1,\sigma}\vu) -\mathcal{R}(\vu)}\leq C_M\gamma_1(\sigma), 
       \end{equation}
        holds uniformly for  all $\vu\in A_M\subset \xs{1} $, 
       with $A_M:=\set{\vx{1}\in \xs{1}:\mathcal{R}(\vx{1})\leq M} .$
 \item There exists an $\mathcal{R}$-minimizing solution $\udag$ of \eqref{eq:linear}. 
 \item $\vy_\sigma, \vy_\sigma^\delta \in \xs{2,}{_\sigma}$ satisfy
       \begin{equation} \label{eq:tandata}
        \norm{\vy_\sigma - \vy_\sigma^\delta}_{2,\sigma} =\mathcal{O}(\delta)\, \text{ and }
        \norm{(T_{2,\sigma})^{-1} \vy_\sigma- \vy}_2 = \mathcal{O}(\gamma_2(\sigma))\,,
        \end{equation}
        where $\gamma_2(\sigma) \to 0$ for $\sigma \to 0$.
\end{itemize}
\end{assu}
\begin{rem}[On Assumption \ref{assu:Surface}]
For surfaces which can not be parametrized by a single domain (such as a sphere), we assume that the domain 
can be covered by patches, and the above assumption has to be satisfied on every patch.
\end{rem}

\begin{rem}
\eqref{eq:delta}, \eqref{eq:C} and \eqref{eq:tandata} actually implies that 
\begin{equation}
\label{eq:tandataII}
\max\set{\norm{(T_{2,\sigma})^{-1}\vy_\sigma^\delta - \vy}_2, \norm{\vy_\sigma^\delta - (T_{2,\sigma})\vy}_{2,\sigma}} = 
\mathcal{O}(\gamma_2(\sigma) + \delta).
\end{equation}
\end{rem}
For given data $\vy_\sigma^\delta$ we consider the regularization strategy consisting in minimization of the Tikhonov functional 
for $\vu_\sigma \in \xs{1,}{_\sigma}$,
\begin{equation}
 \label{eq:Tik_sigma}
 \mathcal{T}_{\alpha,\sigma,\vy_\sigma^\delta}(\vu_\sigma) 
 := \norm{F_\sigma \vu_\sigma -\vy_\sigma^\delta}_{2,\sigma}^2 + \alpha \mathcal{R}_\sigma(\vu_\sigma).
\end{equation}

\begin{thm}
\label{thm:Convergence}
Let the Assumptions \ref{assu:ori}, \ref{assu:Surface} hold. 
Moreover, assume that $\alpha:=\alpha(\sigma,\delta) \to 0$ for $\sigma \to 0$, and $\delta \to 0$ 
and that  
\begin{equation} \label{eq:parameter}
 \frac{\rho^2}{\alpha} \to 0, \; \frac{\gamma_2^2}{\alpha} \to 0 , \; \frac{\delta^2}{\alpha}\to 0 \text{ and }  \gamma_1  \to 0,
 \text{ for } \sigma,\delta\;\to 0.
\end{equation}
Let $(\vu_k:=\vu_{\alpha_k,\sigma_k,\delta_k})$ be a sequence of minimizers of 
$\mathcal{T}_k:=\mathcal{T}_{\alpha_k,\sigma_k,\vy_k}$, defined in \eqref{eq:Tik_sigma}, with 
$\vy_k:=\vy_{\sigma_k}^{\delta_k}$ satisfying \eqref{eq:tandata}. In addition we denote by 
\begin{equation}
 \label{eq:hatTk}
 \check{\vu}_k = T_{1,k}^{-1}\vu_k \in  \xs{1}
\end{equation}
the pullback of $\vu_k$ onto $\surface_1$, and $T_{i,k}:=T_{i,\sigma_k}$ for $i=1,2$ and $k\in \N.$
\begin{itemize}
\item We use the abbreviations $\rho_k:=\rho(\sigma_k)$, 
      $\gamma_{i,k}:=\gamma_i(\sigma_k)$ for $i=1,2$ and $\alpha_k := \alpha(\sigma_k,\delta_k)$. 
      Then the limit of every convergent subsequence of $(\check{\vu}_k)$ is an $\mathcal{R}$-minimizing     
      solution.
\item Moreover, assuming the source condition \eqref{eq:sourcecondition} holds,    
      then, the $\mathcal{R}$-minimizing solution has the convergence rate by Bregman-distance
      \begin{equation*}
      B_{\mathcal{R}}(\check{\vu}_k,\udag)=
      \mathcal{O}\left(\max\set{\frac{\rho_k^2}{\alpha_k},\frac{\gamma_{2,k}^2}{\alpha_k},\frac{\delta_k^2}{\alpha_k},\gamma_{1,k}}\right).
      \end{equation*}
\end{itemize}
\end{thm}
\begin{proof}
\begin{itemize}
\item To prove the first item let $\vu_k$ denote the minimizer of $\mathcal{T}_k$. Moreover, we use the abbreviations 
      $F_k =F_{\sigma_k}$, $\mathcal{R}_k=\mathcal{R}_{\sigma_k}$, and $\udag_k = T_{1,k} \udag$.
      
      Then, according to the definition of a minimizer, we have \eqref{eq:minimizer}
      \begin{equation} \label{eq:minimizer}
      \norm{F_k \vu_k-\vy_k}_{2,k}^2+\alpha_k \mathcal{R}_k(\vu_k) 
      \leq \norm{F_k \udag_k-\vy_k}_{2,k}^2+\alpha_k\mathcal{R}_k(\udag_k) .
       \end{equation} 
      Since the vector field $\udag$ solves $F \udag =\vy$, 
    and \eqref{eq:transform} in Assumption \ref{assu:Surface} and Equations \eqref{eq:rhosigma}, \eqref{eq:tandataII} hold, it follows that
      \begin{equation}
       \label{eq:1}
      \begin{aligned}
      \norm{F_k \udag_k-\vy_k}_{2,k}^2 
      & =  \norm{T_{2,k} T_{2,k}^{-1}(F_k \udag_k-\vy_k)}_{2,k}^2 \\
      & \leq C \norm{T_{2,k}^{-1}(F_k \udag_k-\vy_k)}_2^2 \\
      & =   C \norm{T_{2,k}^{-1}F_k T_{1,k}\udag-F \udag +\vy-T_{2,k}^{-1}\vy_k }_2^2\\
      & = \mathcal{O}(\rho_k^2+\gamma_{2,k}^2+\delta_k^2),
      \end{aligned}
      \end{equation}
      
      Applying  \eqref{eq:rsigma}, we find 
      \begin{equation}
      \label{eq:3}
      \begin{aligned}
      \biggl| \mathcal{R}_k(\udag_k) - \mathcal{R}(\udag) \biggr| & = \mathcal{O}(\gamma_{1,k}).
      \end{aligned}
      \end{equation}
      Using the estimates \eqref{eq:1} and \eqref{eq:3} in \eqref{eq:minimizer} shows that 
      \begin{equation}
      \label{eq:4}
      \begin{aligned}
         ~& \norm{F_k \vu_k-\vy_k}_{2,k}^2 +\alpha_k \abs{ \mathcal{R}_k (\vu_k) -\mathcal{R}(\udag)}\\
         =& \mathcal{O}\left( \max\set{\rho_k^2,\gamma_{2,k}^2,\delta_k^2, \alpha_k\gamma_{1,k}} \right).
      \end{aligned}
      \end{equation}
      From \eqref{eq:parameter} and \eqref{eq:4} it follows that $\set{\mathcal{R}_k(\vu_k): k \in \N}$ is 
      uniformly bounded, and consequently it follows 
      from \eqref{eq:rsigma} that
      \begin{equation}
      \label{eq:5}
       \biggl| \mathcal{R}(\check{\vu}_k)- \mathcal{R}_k(\vu_k) \biggr| = \mathcal{O}(\gamma_{1,k}).
      \end{equation}     
      Since $\check{\vu}_k: = T_{1,k}^{-1} \vu_k$, it follows from 
     \eqref{eq:itransform}, \eqref{eq:C}, \eqref{eq:4} and \eqref{eq:5} that 
	\begin{equation} \label{eq:NeuSch}
	\begin{aligned}
	     & \norm{T_{2,k}^{-1} F_k T_{1,k} \check{\vu}_k -T_{2,k}^{-1} \vy_k}_2^2   + 
	        \alpha_k \left( \mathcal{R}(\check{\vu}_k) - \mathcal{R}(\udag) \right)\\	     
	\leq & C\left(\norm{F_k \vu_k - \vy_k}_{2,k}^2  \right) + 
	        \alpha_k  \biggl| \mathcal{R}_k(\vu_k) - \mathcal{R}(\udag) \biggr|
	        + \alpha_k \biggl| \mathcal{R}(\check{\vu}_k) - \mathcal{R}_k(\vu_k) \biggr|\\ 
	\leq & (1+C)\left(\norm{F_k \vu_k-\vy_k}_{2,k}^2+  
	       \alpha_k \biggl| \mathcal{R}_k(\vu_k) - \mathcal{R}(\udag) \biggr|\right) 
	       + \alpha_k \biggl| \mathcal{R}(\check{\vu}_k) - \mathcal{R}_k(\vu_k) \biggr|\\
	= & \mathcal{O}\left(  \max\set{\rho_k^2,\gamma_{2,k}^2,\delta_k^2, \alpha_k\gamma_{1,k}} \right). 
	\end{aligned}
	\end{equation}
	From the assumptions on the parameters \eqref{eq:parameter} it follows after division of the inequality by $\alpha_k$ and 
	taking the limit $k \to \infty$ afterwards that
	\begin{equation*}
        \limsup_k \mathcal{R}(\check{\vu}_k) \leq \mathcal{R}(\udag),
        \end{equation*}
        and consequently also 
        \begin{equation} \label{eq:dataconverge}
        \lim_k \norm{T_{2,k}^{-1} F_k T_{1,k} \check{\vu}_k- T_{2,k}^{-1} \vy_k}_2^2 = 0.
        \end{equation}
        It follows then from \eqref{eq:dataconverge} and \eqref{eq:rhosigma} that 
        $\norm{F \check{\vu}_k}$ is uniformly bounded.
        
        Similarly as in \cite[Theorem 3.26]{SchGraGroHalLen09} it can be seen that for some fixed $\alpha_0$, there is $N_0\in \N^+$ the set
        $\set{\mathcal{T}_{\alpha_0,\vy}(\check{\vu}_k):k \geq N_0}$ is uniformly bounded. 
        Then by the weak sequential compactness of the level sets of $\mathcal{T}_{\alpha_0,\vy}$ (Assumption \ref{assu:ori}), we have $\set{\check{\vu}_k}$ is bounded in $\xs{1}$.
        Thus there is a weakly convergent subsequence of $\set{\check{\vu}_k}$, for which we denote the 
        weak limit by $\bar{\vu}\in \xs{1}$.

        Using \eqref{eq:dataconverge} it follows that 
	\begin{equation*}
	\begin{aligned}
	  & \norm{F \bar{\vu}-\vy}_2^2 \\
	 =& \norm{F \bar{\vu}- T_{2,k}^{-1} F_kT_{1,k}\bar{\vu} + T_{2,k}^{-1} F_kT_{1,k}\bar{\vu}-T_{2,k}^{-1} \vy_k +T_{2,k}^{-1} \vy_k -\vy}_2^2 \to  0\,.
	\end{aligned}
	\end{equation*}
	This in particular shows that $\bar{\vu}\in \xs{1}$ solves \eqref{eq:linear}.
        The weakly lower semi-continuity of the functional $\mathcal{R}(\cdot)$ implies that 
        \begin{equation*}
         \mathcal{R}(\bar{\vu})\leq \liminf_k \mathcal{R}(\check{\vu}_k) \leq \limsup_k \mathcal{R}(\check{\vu}_k) \leq   \mathcal{R}(\udag),
        \end{equation*}
        which because $\udag$ is an $\mathcal{R}$- minimizing solution tells us that $\mathcal{R}(\bar{\vu})=\mathcal{R}(\udag)$, and thus 
        $\bar{\vu}$ is also an $\mathcal{R}$- minimizing solution of \eqref{eq:linear}.
\item To prove the convergence rate result we reconsider \eqref{eq:NeuSch} and the family of operators  
      $\{T_{2,k}^{-1}F_k T_{1,k}: k \in \N\}$. Using \eqref{eq:rhosigma} in Assumption \ref{assu:Surface},
      it follows that 
      \begin{equation*}
      \norm{T_{2,k}^{-1}F_k  T_{1,k} -F} = \mathcal{O}\left(\rho_k\right).
      \end{equation*}
      Moreover, from \eqref{eq:tandataII} it follows that 
      \begin{equation*}
      \norm{T_{2,k}^{-1} \vy_k-\vy}_2 = \mathcal{O}(\gamma_{2,k}+\delta_k).
      \end{equation*}
      We can apply the results of \cite[Theorem 2.6]{PoeResSch10} with the triangle inequality to taking care of some additional error terms on the right side of 
      \eqref{eq:NeuSch}, then we get
      \begin{equation*}
     B_{\mathcal{R}}(\check{\vu}_k,\udag) = 
      \mathcal{O}\left(\max\set{\frac{\rho_k^2}{\alpha_k},\frac{\gamma_{2,k}^2}{\alpha_k},\frac{\delta_k^2}{\alpha_k},\gamma_{1,k}}\right).
      \end{equation*}
\end{itemize}
\end{proof}
\begin{rem} 
We note that if the parameters are chosen in the following way, $\rho(\sigma)=\mathcal{O}(\gamma_1(\sigma))=\mathcal{O}(\gamma_2(\sigma))=\mathcal{O}(\delta)$ (for instance by appropriate discretization)
and $\alpha(\sigma,\delta)=\mathcal{O}(\delta)$ , then we 
derive the standard convergence rates
\begin{equation*}
B_{\mathcal{R}}(\check{\vu}_k,\udag) =   \mathcal{O}\left(\delta_k\right).
\end{equation*}
Especially, if we choose $\mathcal{R}(\cdot)=\norm{\cdot}_1^2,$ then we have
\[B_{\mathcal{R}}(\check{\vu}_k,\udag) =\norm{\check{\vu}_k-\udag}_1^2 = \mathcal{O}\left(\delta_k\right).\]
\end{rem}
Theorem \ref{thm:Convergence} is a generalization of Theorem \ref{thm:directreg}. Actually Theorem \ref{thm:directreg} is a trivial case of Theorem \ref{thm:Convergence} when $\sigma \equiv 0$.

In the following, we shortly discuss a case example in which $\xs{i}\; (i=1,2)$ are the Sobolev spaces on surfaces (see \cite{Heb99} for instance). 
\begin{exa}
\label{exa:Sobolev}
Let $\xs{i,}{_\sigma}:=W^{k_i,2}(\surface_{i,\sigma},\R^{r_i})$, for $i=1,2$, $k_i \in \N$ and $r_i \in \R^+$,
\begin{equation}
\label{eq:Sobolev}
m_{i,\sigma}\in  W^{k_i,\infty}(\Omega_i,\surface_{i,\sigma}) \; \text{ and }\;
m_{i,\sigma}^{-1} \in  W^{k_i,\infty}(\surface_{i,\sigma},\Omega_i).
\end{equation}
We have $T_{i,\sigma}$ and $T_{i,\sigma}^{-1}$ are bounded.
Moreover, let
\begin{equation}
\label{eq:gamma}
\gamma_i(\sigma):=\norm{\parametrization_{i,\sigma} - \parametrization_i}_{W^{k_i,\infty}(\Omega_i)}\;
\end{equation}
and assume that $\gamma_i : \R^+ \to \R$ is uniformly bounded, monotonically increasing,
and satisfies the convergence $\gamma_i(\sigma) \to 0$ as $\sigma \to 0$\,. 
For a general operator equation \eqref{eq:linear} of which the data satisfies \eqref{eq:delta},
we consider its Tikhonov regularization approximation \eqref{eq:Tik} with $\mathcal{R}(\cdot)=\norm{\cdot}_{\xs{1}}^2$.
Then \eqref{eq:C}, \eqref{eq:rsigma} and \eqref{eq:tandata} are also fulfilled.
\end{exa}
\begin{rem}
In particular if $k_i=2$ and $m_{i,\sigma}$ is a uniform cubic spline approximation with grid size $h<1$ of $m_i$ which is $C^2$-smooth, then 
the condition \eqref{eq:Sobolev} holds and the estimate \eqref{eq:gamma} can have an explicitly form, see for instance \cite{AhlNil63}, that is
\begin{equation}
\label{eq:gamma_h}
\gamma_i(h):=\norm{\parametrization_{i,h} - \parametrization_i}_{W^{2,\infty}(\Omega_i)} = \mathcal{O}(h) \;.
\end{equation}
\end{rem}
\begin{rem}
Note the condition \eqref{eq:rhosigma} has to be further checked for individual operators $F$ and $F_\sigma$.
The source condition \eqref{eq:sourcecondition} may require more smoothness on the non-disturbed surfaces $\surface_1$ and $\surface_2$.
For example: if $F:W^{1,2}(\surface_1)\rightarrow L^2(\surface_1)$ is the embedding operator,
and $\mathcal{R}(\cdot)=\abs{\cdot}_{W^{1,2}(\surface_1)}$,
then condition \eqref{eq:sourcecondition} reads as
\[ \Delta_{\surface} \udag \in L^2(\surface_1), \]
where $\Delta_{\surface}$ denotes the Laplace--Beltrami operator.
It asks that $\parametrization_1 \in W^{2,\infty}(\Omega_1)$,
which is $one$ order higher regularity than $W^{1,\infty}(\Omega_1)$ given in \eqref{eq:Sobolev}.
\end{rem}

\section{Application to recover vector fields in ambient spaces}
\label{sec:vectorfields}
In this section, we are studying an ill-posed operator equation, of which the solution
is a functions with range in the vector bundle over a surface.
A typical case is like a \emph{tangent vector field} $\widetilde{\vu}:\surface \rightarrow \mathcal{T}\surface$ (see Appendix \ref{appendix:background}).
As a consequence, in case of surface perturbations, also the range of the function $\widetilde{\vu}$ is perturbed 
and we get the approximation
\[  \widetilde{\vu_\sigma} :\surface_\sigma \rightarrow \mathcal{T}\surface_\sigma. \]
To take this into account we consider vector fields represented by the basis in \emph{ambient space} $\R^{d+1}$ of the surface $\surface$,
which consist of vector valued functions $\vu:\surface \rightarrow \R^{d+1}$.

We start by introducing some appropriate function spaces. The relevant geometric notations are summarized in 
Table \ref{tab:Geo_Notation} from the Appendix \ref{appendix:background}.

\subsection{Spaces of functions with range in the vector bundle}
\label{sec:spaces}
Before introducing the function spaces we outline a basic assumption first:
\begin{assu}
\label{ass:general}
Let $\surface\in \R^{d+1}$ be a $d$-dimensional, differentiable surface, such that the \emph{surface gradient} 
(cf. Definition \ref{de:gradient}) of the unit normal vector $\normal$ of the surface satisfies 
\begin{equation}
 \label{eq:smooth_ma}
 \norm{\nabla_\surface \normal}_{L^\infty(\surface)} \leq C_c,
\end{equation}
for some appropriate constant $C_c$. Here 
\[\norm{\nabla_\surface \normal}_{L^\infty(\surface)}:=\sup_{x\in \surface} \set{\abs{\nabla_\surface \normal(x)}}\]
and $\abs{\cdot}$ denotes the Frobenius norm of a matrix.
\end{assu}
In fact, \eqref{eq:smooth_ma} is a uniform bound on the extrinsic curvature of the surface $\surface$.
The surface gradient operator $\nabla_\surface \cdot $ should not be confused with the covariant derivative $\widetilde{\nabla} \cdot$ 
(see Definition \ref{de:covariant} in Appendix). Note that the latter is only defined for functions with range in the tangent bundle and does not involve the metric of the surface.
$\gradientM \normal$ is sometimes referred as \emph{shape operator} in the literature.

\begin{defi}
\label{def:spaces}
Let $\leb{2}$ denote the space of square integrable, scalar, vector, and matrix valued functions on $\surface$.
The inner products and norms are defined by
\begin{equation*}
   \inner{U_1}{U_2}_{\leb{2}} := \int_\surface U_1 \cdot U_2\,ds(x),\quad \norm{U}_{\leb{2}}^2 = \inner{U}{U}_{\leb{2}}, 
\end{equation*}
respectively. Here $s(x)$ denotes the $d$-dimensional surface measure. 
Note that if $U_1$, $U_2$ are scalar valued functions, then $\cdot$ denotes multiplication of numbers and for vectors and 
matrices $\cdot$ denotes component wise multiplication. $\abs{\cdot}$ (without any subscript) 
denotes the Euclidean norm of a vector or the Frobenius norm of a matrix.

We define the sets
\begin{equation}
       \label{eq:h1}
       \hspace{-0.02\textwidth}
       \begin{aligned}
        \htm{\surface} &:= 
        \left\{ \tv : \surface \rightarrow \R^{d+1}: \tv(x) \in \mathcal{T}_x\surface,\, \forall x \in \surface,\; 
                    \norm{\tv}_{\htm{\surface}} < \infty \right\},\\
        \hn{\surface} &:= 
        \left\{ \vn : \surface \rightarrow \R^{d+1}: \vn(x) \in \mathcal{N}_x\surface,\, \forall  x \in \surface,\;
                    \norm{\vn}_{\hn{\surface}} < \infty \right\},\\
        \hone{\surface} &:= \left\{ \vu=\tv+\vn: \tv \in \htm{\surface},\,
        \vn \in {\hn{\surface}},\, \norm{\vu}_{\hone{\surface}} < \infty \right\}, \\
        \Hone{\surface} &:= 
        \set{\vu : \surface \rightarrow \R^{d+1} : \norm{\vu}_{\Hone{\surface}} < \infty},
        \end{aligned}
       \end{equation}
       where $\mathcal{T}_x\surface$ and $\mathcal{N}_x\surface$ are the tangent and normal spaces, respectively 
       (see Appendix \ref{appendix:background}).
       
       The associated inner products and norms for 
       \[A\in \set{\htm{\surface}\;,\hn{\surface}\;, \hone{\surface}\;,\Hone{\surface}},\]
       respectively, are defined by
       \begin{equation}
       \label{eq:hnorm}
       \begin{aligned}
       \inner{\vu_1}{\vu_2}_A &:= \inner{\vu_1}{\vu_2}_{\leb{2}}+ \inner{U_1}{U_2}_{\leb{2}},\\
       \norm{\vu_1}_A^2 &:= \inner{\vu_1}{\vu_1}_A \text{ and }
       \abs{\vu_1}_A := \norm{U_1}_{\leb{2}},
       \end{aligned}
       \end{equation}
       where 
       \begin{equation*}
         U_i = \left\{ \begin{array}{rl}
                         \Ptm \gradientM  \vu_i & \text{ for } \vu_i\in A=\htm{\surface},\\
                         \Pn \gradientM  \vu_i & \text{ for } \vu_i\in A=\hn{\surface},\\
                         \Ptm \gradientM (\Ptm \vu_i) + \Pn  \gradientM (\Pn\vu_i) & \text{ for } \vu_i\in A=\hone{\surface},\\
                         \gradientM \vu_i & \text{ for } \vu_i\in A=\Hone{\surface},
                       \end{array} \right.\quad i=1,2,
       \end{equation*}
      where the definitions of $\Ptm$ and $\Pn$ can be found in Appendix \ref{appendix:background} \eqref{eq:pn}.
\end{defi}

In the following we prove that these spaces are in fact Hilbert spaces. Moreover, we prove some equivalent relations to the standard 
Sobolev space $\Hone{\surface}$ which is another way to denote the space $W^{1,2}(\surface,\R^{d+1})$ (c.f. Example \ref{exa:Sobolev}).
\begin{lem}
\label{le:directsum}
Let $\hone{\surface}$, $\htm{\surface}$ and $\hn{\surface}$ as defined in \eqref{eq:h1}.
Then $\hone{\surface}$ is the direct sum of $\htm{\surface}$ and $\hn{\surface}$, that is, 
\begin{equation*}\hone{\surface}=\htm{\surface}\oplus \hn{\surface}.\end{equation*}
\end{lem}
\begin{proof}
First we prove that $\htm{\surface} + \hn{\surface} = \hone{\surface}$.
By definition $\htm{\surface} + \hn{\surface} \subset \hone{\surface}$. On the other hand
Let $\tv= \Ptm \vu$  and $\vn = \Pn \vu$, then $\vu = \tv + \vn$, and thus 
$\hone{\surface} \subset \htm{\surface} + \hn{\surface}$.

It remains to prove that $\htm{\surface}$ and $\hn{\surface}$ are orthogonal with respect 
to $\norm{\cdot}_{\hone{\surface}}$. For every $\tv \in\htm{\surface}$ and $\vn\in \hn{\surface}$, we have
\begin{equation*}
  \inner{\tv}{\vn}_{\hone{\surface}} =\inner{\Ptm \tv(x)}{ \underbrace{\Ptm \vn(x)}_{=0}}_{\htm{\surface}}+
                             \inner{\underbrace{\Pn \tv(x)}_{=0}}{\Pn \vn(x)}_{\hn{\surface}}.
\end{equation*}
Thus, also the orthogonality is proven.
\end{proof}
In the following we prove auxiliary results, which show embeddings of the space $\hone{\surface}$.

\begin{lem}
\label{le:proj}
Let Assumption \ref{ass:general} hold. Then 
the projection operators $\Ptm$ and $\Pn$ are continuous from the standard Sobolev space $\Hone{\surface}$ 
to $\htm{\surface}$ and $\hn{\surface}$, and in particular satisfy:
 \begin{equation}
  \label{eq:pr_l2}
  \max\set{\norm{\Ptm \vu}_{\hone{\surface}}^2,\norm{\Pn \vu}_{\hone{\surface}}^2}\leq (1+(1+C_c)^2)\abs{\vu}^2_{\Hone{\surface}}.
 \end{equation}
 \end{lem}
\begin{proof}
From the definition of \eqref{eq:hnorm} it follows that 
\begin{equation*}
 \abs{\Ptm \vu}_{\hone{\surface}} = \norm{\Ptm  \gradientM(\Ptm \vu) + \Pn  \gradientM(\Pn \Ptm \vu)}_{\leb{2}}
 = \norm{\Ptm  \gradientM(\Ptm \vu)}_{\leb{2}}.
\end{equation*}
Then, by using Lemma \ref{le:Pro_gradient}, and triangle inequality it follows that
\begin{equation}
\label{eq:est0}
\begin{aligned}
 \abs{\Ptm \vu}_{\hone{\surface}} 
 =& \norm{\Ptm\gradientM \vu-(\normal^T\vu)\gradientM \normal}_{\leb{2}}\\
 \leq & \norm{\Ptm\gradientM \vu}_{\leb{2}} + \norm{(\normal^T\vu)\gradientM \normal}_{\leb{2}}\;.
\end{aligned}
\end{equation}
Because, for $x \in \surface$, the Frobenius-norm of $(\normal^T(x)\vu(x))\gradientM \normal(x)$ satisfies
\begin{equation*}
\abs{(\normal^T(x)\vu(x))\gradientM \normal(x)}
=  \abs{\normal^T(x)\vu(x)} \abs{\gradientM \normal(x)}\,,
\end{equation*}
it follows from \eqref{eq:smooth_ma} and Cauchy-Schwarz inequality on $\R^{d+1}$ that
\begin{equation}
\label{eq:est}
 \begin{aligned}
  \norm{(\normal^T\vu)\gradientM \normal}_{\leb{2}}^2 &= \int_\surface \abs{(\normal^T\vu)\gradientM \normal}^2 ds(x)\\
  &\leq  \norm{\nabla_\surface \normal}^2_{L^\infty(\surface)}\int_\surface \abs{\inner{\normal}{\vu}}^2\,ds(x)\\
  &\leq  C_c^2 \int_\surface \abs{\vu}^2\,ds(x)\;.
 \end{aligned}
\end{equation}
Using \eqref{eq:est} and \eqref{eq:AA} in \eqref{eq:est0} then shows that
\begin{equation*}
\abs{\Ptm \vu}_{\hone{\surface}} \leq  
\norm{\gradientM \vu}_{\leb{2}} + C_c\norm{\vu}_{\leb{2}}
\leq (1+ C_c)\norm{\vu}_{\Hone{\surface}}.
\end{equation*} 
From this it directly follows that
\begin{equation*}
\norm{\Ptm \vu}_{\hone{\surface}}^2 =\norm{\Ptm \vu}_{\leb{2}}^2+ \abs{\Ptm \vu}_{\hone{\surface}}^2\leq (1+(1+ C_c)^2)\norm{\vu}_{\Hone{\surface}}^2.
\end{equation*}
Moreover, from the definition of $\Hone{\surface}$, \eqref{eq:hnorm}, it follows that
\begin{equation}
\label{eq:I}
  \abs{\Pn \vu}_{\hone{\surface}} = \norm{\Pn \gradientM \Pn \vu}_{\leb{2}}\;.
\end{equation}
Then, by using Lemma \ref{le:Pro_gradient}, triangle inequality, and the fact that the Frobenius norm is sub-multiplicative
it follows that 
\begin{equation}
\label{eq:II}
 \begin{aligned}
 \norm{\Pn \gradientM \Pn \vu}_{\leb{2}} &= \norm{\Pn \gradientM \vu + n\vu^T \gradientM \normal}_{\leb{2}}\\
 &\leq \norm{\Pn \gradientM \vu}_{\leb{2}} + 
       \sqrt{\int_\surface \abs{\normal \vu^T \gradientM  n}^2 \,d s(x)}\\
 &\leq \norm{\Pn \gradientM \vu}_{\leb{2}} + 
       \sqrt{\int_\surface \abs{\normal \vu^T }^2 \abs{\gradientM n}^2 \,d s(x)}\\      
 &\leq \norm{\Pn \gradientM \vu}_{\leb{2}} + C_c 
       \sqrt{\int_\surface \abs{\normal \vu^T }^2 \,d s(x)}\;.
\end{aligned}
\end{equation} 
By using that for a matrix $\normal \vu^T$ the spectral and the Frobenius norm are identical and satisfy 
$\abs{\normal \vu^T} = \abs{\normal} \abs{\vu} = \abs{\vu}$ \cite{GolVan96}, we get from \eqref{eq:II} 
\begin{equation}
\label{eq:III}
 \norm{\Pn \gradientM \Pn \vu}_{\leb{2}} \leq \norm{\gradientM \vu}_{\leb{2}} + C_c \norm{\vu}_{\leb{2}}\;.
\end{equation} 
Combining \eqref{eq:I} and \eqref{eq:III}, it then follows that 
\begin{equation*}
  \abs{\Pn \vu}_{\hone{\surface}} \leq   (1+ C_c)\norm{\vu}_{\Hone{\surface}}.
\end{equation*}
In summary, we have
\begin{equation*}
\norm{\Pn \vu}^2_{\hone{\surface}} =\norm{\Pn \vu}^2_{\leb{2}}+ \abs{\Pn \vu}^2_{\hone{\surface}}\leq (1+(1+ C_c)^2)\norm{\vu}^2_{\Hone{\surface}}.
\end{equation*}
\end{proof}
Moreover, we have the equivalence between $\hone{\surface}$ and $\Hone{\surface}$.
\begin{lem}
\label{le:normbound}
Let Assumption \ref{ass:general} hold. Then $\norm{\cdot}_{\hone{\surface}}$ and $\norm{\cdot}_{\Hone{\surface}}$
(see \eqref{eq:h1}) are equivalent on $\hone{\surface}$.
\end{lem}
\begin{proof}
We first show that 
\begin{equation*}
\norm{\vu}_{\hone{\surface}} \leq  \sqrt{2(1+(1+ C_c)^2)}\norm{\vu}_{\Hone{\surface}},\quad \forall \vu \in \hone{\surface}. 
\end{equation*}
This follows from Assumption \ref{ass:general}, Definition \ref{def:spaces} and Lemma \ref{le:proj} together with
\begin{equation*}
\norm{\vu}^2_{\hone{\surface}}=\norm{\Ptm \vu}^2_{\hone{\surface}}+\norm{\Pn \vu}^2_{\hone{\surface}}\leq 2(1+(1+ C_c)^2)\norm{\vu}^2_{\Hone{\surface}}.
\end{equation*}
Now, we show that
\begin{equation*}
\norm{\vu}_{\Hone{\surface}} \leq \sqrt{(4C_c^2+2)}\norm{\vu}_{\hone{\surface}}. 
\end{equation*}
Since $\Ptm \vu+\Pn \vu=\vu$, and by orthogonality of $\Ptm \vu$ and $\Pn \vu$ in $\leb{2}$ it follows that
\begin{equation*}
\norm{\vu}^2_{\leb{2}}=\norm{\Ptm \vu}^2_{\leb{2}}+\norm{\Pn \vu}^2_{\leb{2}}.
\end{equation*}
Then, from Lemma \ref{le:Pro_gradient}, it follows that
 \begin{equation*}
 \Ptm\gradientM (\Ptm \vu) + \Pn\gradientM (\Pn \vu)=\gradientM \vu-\normal^T \vu\gradientM \normal +\normal \vu^T\gradientM \normal.
 \end{equation*}
and therefore from the definition of $\hone{\surface}$ \eqref{eq:hnorm} it follows that
\begin{equation*}
\abs{\vu}^2_{\hone{\surface}}=\norm{\gradientM \vu -\normal^T \vu\gradientM \normal +\normal \vu^T\gradientM \normal}^2_{\leb{2}}.
\end{equation*}
Then by triangle inequality, and using the estimates used in \eqref{eq:II} in Lemma \ref{le:proj}, it follows that
\begin{eqnarray*}
 \norm{\gradientM \vu}_{\leb{2}} 
&\leq & \norm{\gradientM \vu -\normal^T \vu\gradientM \normal +\normal \vu^T\gradientM \normal}_{\leb{2}}  \\
& & \quad + 
\norm{\normal^T \vu\gradientM \normal} + \norm{\normal \vu^T\gradientM \normal}_{\leb{2}} \\
&\leq &\abs{\vu}_{\hone{\surface}}+ 2C_c\norm{\vu}_{\leb{2}}.
\end{eqnarray*}
Hence, $\norm{\gradientM \vu}_{\leb{2}}^2\leq  \abs{\vu}_{\hone{\surface}}^2+ 4C_c^2\norm{\vu}_{\leb{2}}^2,$
and we get the estimate \[\norm{\vu}^2_{\Hone{\surface}}\leq  (4C_c^2+2)\norm{\vu}^2_{\hone{\surface}}.\]
\end{proof}
\begin{rem}[on Lemma \ref{le:normbound}]
\label{rem:spaces}
Note that the equivalence of the norms $\norm{\cdot}_{\Hone{\surface}}$ and $\norm{\cdot}_{\hone{\surface}}$ of the 
spaces $\hone{\surface}$ and $\Hone{\surface}$ is based on the uniform boundedness of the curvature of $\surface$ 
(cf. Assumption \ref{ass:general}).

While Lemma \ref{le:normbound} guarantees equivalence of the norms, this does not induce equivalence of the seminorms 
$\abs{\cdot}_{\Hone{\surface}}$ and $\abs{\cdot}_{\hone{\surface}}$.
\end{rem}
With the discussion above, we can conclude that the spaces introduced in Definition \ref{def:spaces} are actually Hilbert spaces.
\begin{thm}
The spaces $\hone{\surface}$, $\htm{\surface}$ and $\hn{\surface}$ associated with the inner products $\inner{\cdot}{\cdot}_{\hone{\surface}}$, $\inner{\cdot}{\cdot}_{\htm{\surface}}$ and $\inner{\cdot}{\cdot}_{\hn{\surface}}$, respectively, are Hilbert spaces.
\end{thm}
\begin{proof}
The first assertion follows by the equivalence of the Sobolev space $\Hone{\surface}$ and $\hone{\surface}$. 
The second and the third ones are fulfilled because they are orthogonal subspaces of $\hone{\surface}$ and complement 
each other there. Hence $\htm{\surface}$ and $\hn{\surface}$ 
are closed subspaces and thus they are Hilbert spaces.
\end{proof}
       
\begin{rem}         
We emphasize that all the norms in the Definition \ref{def:spaces} are independent of the parametrization. 
The independence follows by the invariance of the surface gradient $\gradientM$ (cf. Lemma \ref{le:invariant}), 
as well as the projection operators $\Ptm$ and $\Pn$, with respect to the parametrizations. 

We also note that $\abs{\tv}_{\htm{\surface}}$ is a seminorm, which can be expressed via the covariant derivative 
(cf. Definition \ref{de:covariant}): 
According to \eqref{eq:cov_m_t}, we have 
       \begin{equation*} 
       \begin{aligned}
       \abs{\tv}_{\htm{\surface}}^2 &= 
       \int_\surface \abs{\Ptm(x) \gradientM  \tv(x)}^2 d s(x)\\
       &\underbrace{=}_{\eqref{eq:cov_m_t}} 
       \int_\surface \abs{(\widetilde\nabla \tv(x) (\partial m(\zeta))^\dagger)}^2 d s(x)\\
       &=
       \int_\surface \trace\left(\left(\widetilde\nabla \tv(x)  (\partial m(\zeta))^\dagger)\right)\left(\widetilde\nabla \tv(x)  (\partial m(\zeta))^\dagger)\right)^T\right) d s(x)\\
       &\underbrace{=}_{\eqref{eq:projectiona}}
       \int_\surface  \trace\left( \left(\widetilde\nabla \tv(x)\right)g^{-1} \left(\widetilde\nabla \tv(x)\right)^T \right)  d s(x),
       \end{aligned}
       \end{equation*} 
       where the $\trace$ is the trace of a matrix. 
\end{rem}
\subsection{Regularization theory for vector fields}
We proceed to discuss the regularization theory with respect to the specific spaces $\mathcal{X}_1(\surface_1)=\hone{\surface_1}$ and $\mathcal{X}_2(\surface_2)=\lebtwo{2}$, which are useful for many applications.
In practice, it is common to consider the following type of regularization functional
\begin{equation}
\label{eq:r_functional}
\mathcal{R}(\vu):=\int_{\surface_1} R(\Ptm\gradientM \Ptm\vu + \Pn\gradientM \Pn \vu)
\end{equation}
where $R(\cdot)$ is a real, non-negative, local Lipschitz and convex function, and thus 
\[\mathcal{R}:\hone{\surface_1} \rightarrow [0,+\infty]\]
is proper, convex and weakly lower semi-continuous.
Since the spaces are fixed, we can have a precise smoothness characterization on surfaces $\surface_{1,\sigma}$ and $\surface_{2,\sigma}$.
\begin{coro}
\label{coro:Surface_vector}
For every $i=1,2$, and every $\sigma \geq 0$, let \eqref{eq:Sobolev} and \eqref{eq:gamma} hold for $k_1=2,$ and $k_2=0,$ 
and let the regularization functional $\mathcal{R}$ be given as in \eqref{eq:r_functional}, then the estimates \eqref{eq:transform}--\eqref{eq:C}, \eqref{eq:rsigma} and \eqref{eq:tandata} hold for $\hone{\surface_{1,\sigma}}$ and $L^2(\surface_{2,\sigma})$, 
as well the curvature estimate \eqref{eq:smooth_ma} holds for $\surface_{1,\sigma}$.
\end{coro}
\begin{rem}
We point out that in order to have compatibility between the smoothness of surfaces and the regularity of the spaces, we ask the parametrization map $m_{1,\sigma}\in W^{2,\infty}(\Omega_1)$ on the surfaces for spaces $\hone{\surface_{1,\sigma}}$, but only ask $m_{1,\sigma}\in W^{1,\infty}(\Omega_1)$ for spaces $\Hone{\surface_{1,\sigma}}$ (c.f. Example \ref{exa:Sobolev}).
\end{rem}

In company with Assumption \ref{assu:ori} and Corollary \ref{coro:Surface_vector}, we can apply the results from Theorem \ref{thm:directreg} and Theorem \ref{thm:Convergence} to the regularization of problem \eqref{eq:linear} associated with vector fields represented in ambient coordinates.
 
\subsection{Recovering tangent vector fields}
We restrict ourselves to solving problem \eqref{eq:linear}, 
where the operator $\widetilde{F}$ is applied to functions with range in the tangent bundle only.
We rewrite \eqref{eq:linear} as \eqref{eq:linear_tan} in this particular case
\begin{equation}
\label{eq:linear_tan}
\widetilde{F}\tv =\vy, \;\text{ for } \; \tv\in \widetilde{\xs{1}}.
\end{equation}
The ambient approach requires to extend the operator $\widetilde{F}$ and the associated Hilbert 
space $\widetilde{\xs{1}}$ to the space $\xs{1}$ which is for vector fields represented in ambient coordinates.

For the sake of simplicity, we select the concrete representation
\begin{equation*}
  \widetilde{\xs{1}}= \htm{\surface_1} \text{ with } 
 \widetilde{\mathcal{R}}(\cdot) = \abs{\cdot}_{\htm{\surface_1}}^2, \mathcal{R}(\cdot) = \abs{\cdot}_{\hone{\surface_1}}^2, \text{ and }
 \xs{2} = \lebtwo{2}.
\end{equation*}
For all $\tv \in \htm{\surface_1}$ we have $\Pn \tv=0$, hence it follows that 
\begin{equation*}
 \abs{\tv}_{\htm{\surface_1}}^2 = 
 \int_\surface \abs{\Ptm(x) \gradientM \Ptm(x) \vu(x)}^2\,ds(x) = \abs{\tv}_{\hone{\surface_1}}^2.  
\end{equation*}
Note that $\abs{\cdot}_{\hone{\surface_1}}^2$ is well-defined on both $\hone{\surface_1}$ and $\htm{\surface_1}$, 
and thus it can be considered as an extension of the functional $\abs{\cdot}_{\htm{\surface_1}}^2$ to ambient space.

In the following we define an ambient operator $F$ of $\widetilde{F}$ 
and the associated ambient operator equation, respectively.
\begin{defi}
\label{def:Operator_e}
The ambient operator $F:\hone{\surface_1}\to \lebtwo{2}$ is a bounded linear operator which 
extends $\widetilde{F}$ from $\htm{\surface_1}$ to $\hone{\surface_1}$. That is
\begin{equation*}
F\vu=\widetilde{F}\vu  \quad \text{for all  }\vu\in \htm{\surface_1}.
\end{equation*}
We consider solving the system of equations
\begin{equation}
\label{eq:linear_e}  \binom{F}{\Pn} \vu = \binom{\vy}{0}.
\end{equation}
We call \eqref{eq:linear_e} the \emph{ambient operator equation} for tangent vector fields.
The second equation of the system ensures that $\vu$ is tangential.
\end{defi}
Tikhonov regularization for solving the ambient operator equation \eqref{eq:linear_e} consists in minimization of the energy functional
\begin{equation}
 \label{eq:Tik_ambint}
 \mathcal{T}_{\alpha,\vy^\delta}(\vu) := \norm{F\vu-\vy^\delta}_{\lebtwo{2}}^2+\norm{\Pn \vu}_{L^2(\surface_1)}^2  + \alpha \abs{\vu}_{\hone{\surface_1}}^2.
\end{equation}
\begin{lem}
\label{le:properties}
Assume that 
$\widetilde{\mathcal{D}}(\widetilde{F})=\htm{\surface_1}$ and that Assumptions \ref{assu:ori} and \ref{ass:general} hold for \eqref{eq:linear_tan}.
Let $F$ be the extended operator introduced in Definition \ref{def:Operator_e}.
Then the following assertions hold:
\begin{enumerate}
\item $\abs{\cdot}_{\hone{\surface_1}}^2: \hone{\surface_1} \to [0,+\infty)$ is a proper, convex, and weakly lower-semi continuous 
      functional satisfying 
      $\mathcal{D} := \mathcal{D}(\abs{\cdot}_{\hone{\surface_1}}^2)\cap \mathcal{D}(F) =\hone{\surface_1}.$
\item $\tilde{\vu}_0 \in \htm{\surface_1}$ is a $\abs{\cdot}_{\htm{\surface_1}}^2$ seminorm minimizing solution of \eqref{eq:linear} if and only if it is a $\abs{\cdot}_{\hone{\surface_1}}^2$ seminorm minimizing solution of \eqref{eq:linear_e}. 
\end{enumerate}
\end{lem}
To prove convergence of the regularization method we still need the following compactness results.

\begin{lem}
\label{le:compact}
Let the same assumptions as in Lemma \ref{le:properties} hold.
Then, for every $\alpha,\theta >0$, the sets 
\begin{equation*}
M_{\alpha,\vy}(\theta):=\set{\vu \in \hone{\surface_1} : \mathcal{T}_{\alpha,\vy}(\vu)\leq \theta}
\end{equation*}
are weakly sequentially compact in $\hone{\surface_1}$.
\end{lem}
\begin{proof}
For every $\vu\in\hone{\surface_1}$, by Definition \ref{def:spaces} and Lemma \ref{le:directsum}, we have 
\[\vu=\Ptm\vu+ \Pn \vu=\tv+\vn ,\]
where $\tv \in \htm{\surface_1}$ and $\vn\in \hn{\surface_1}$. By Definition \ref{def:Operator_e}, the operator $F$ fulfils $\norm{F\vu - \widetilde{F}\tv}_{\lebtwo{2}}^2=\norm{F{\vn}}_{\lebtwo{2}}^2$.
Now, we use the Peter--Paul inequality with $\epsilon>0$ and get
\begin{equation*}
 \begin{aligned}
  \norm{\widetilde{F}\tv -\vy}_{\lebtwo{2}}^2 &= \norm{\widetilde{F}\tv -F\vu  + F\vu -\vy}_{\lebtwo{2}}^2\\ 
  &\leq  (1+\epsilon)\norm{F{\vn}}_{\lebtwo{2}}^2+ (1+\frac{1}{\epsilon})\norm{F\vu -\vy}_{\lebtwo{2}}^2,
 \end{aligned}
\end{equation*}
and consequently it follows that
\begin{equation}\label{eq:compactestimate}
 \begin{aligned}
  & \mathcal{T}_{\alpha,\vy}(\vu)\\
    =  & \norm{F\vu-\vy}_{\lebtwo{2}}^2+\norm{\Pn\vu}_{L^2(\surface_1)}^2 + \alpha\abs{\vu}_{\hone{\surface_1}}^2\\
  \geq & \frac{\epsilon}{1+\epsilon} \norm{\widetilde{F}\tv -\vy}_{\lebtwo{2}}^2 -
          \epsilon \norm{F{\vn}}_{\lebtwo{2}}^2+\norm{\Pn\vu}_{L^2(\surface_1)}^2 +\alpha \abs{\vu}_{\hone{\surface_1}}^2.
 \end{aligned}
\end{equation}
Since $\abs{\vu}_{\hone{\surface_1}}^2=\abs{\tv}_{\htm{\surface_1}}^2+\abs{\vn}_{\hn{\surface_1}}^2$, and let
\[\widetilde{\mathcal{T}}_{\alpha,\vy}(\tv):=\norm{\widetilde{F}\tv -\vy}_{\lebtwo{2}}^2+\alpha \abs{\tv}_{\htm{\surface_1}}^2 \;,\]
it follows from \eqref{eq:compactestimate} that
\begin{equation*}
 \begin{aligned}
  \mathcal{T}_{\alpha,\vy}(\vu) 
  \geq &\frac{\epsilon}{1+\epsilon}\widetilde{\mathcal{T}}_{\alpha,\vy}(\tv) 
        + \frac{\alpha}{1+\epsilon} \abs{\tv}_{\htm{\surface_1}}^2\\
  & \quad + \norm{\vn}_{L^2(\surface_1)}^2 -\epsilon \norm{F{\vn}}_{\lebtwo{2}}^2+\alpha\abs{\vn}_{\hn{\surface_1}}^2.
  \end{aligned}
\end{equation*}
Choosing $\epsilon =\min \set{\frac{\alpha}{2\norm{F}^2},\frac{1}{2\norm{F}^2}}$, we get
\begin{equation*}
 \begin{aligned}
      & \norm{\vn}_{L^2(\surface_1)}^2 -\epsilon \norm{F{\vn}}_{\lebtwo{2}}^2+\alpha\abs{\vn}_{\hn{\surface_1}}^2\\
 \geq & 2\epsilon\norm{F}^2\norm{\vn}_{\hone{\surface_1}}^2-\epsilon\norm{F}^2\norm{\vn}_{\hone{\surface_1}}^2,
 \end{aligned}
\end{equation*}
and subsequently, we have
\begin{equation*}
 \begin{aligned}
 \mathcal{T}_{\alpha,\vy}(\vu) 
 \geq \frac{\epsilon}{1+\epsilon}\widetilde{\mathcal{T}}_{\alpha,\vy}(\tv) 
        + \frac{\alpha}{1+\epsilon} \abs{\tv}_{\htm{\surface_1}}^2+\epsilon\norm{F}^2\norm{\vn}_{\hone{\surface_1}}^2.
 \end{aligned}
\end{equation*}
This estimate shows that every level set of $\mathcal{T}_{\alpha,\vy}$ is uniformly bounded in $\hone{\surface_1}$ and thus 
has a weakly convergent subsequence in $\hone{\surface_1}$. Because of the weak-lower semi-continuity of the norms and seminorms 
and the boundedness of $F$ it follows that the limit is also an element of the level set, which gives the assertion.
\end{proof}

The conditions in Assumption \ref{assu:Surface} and \ref{ass:general} are satisfied by Corollary \ref{coro:Surface_vector}, 
except the estimate \eqref{eq:rhosigma} on operators.
\begin{lem}
\label{le:extended_operator}
Let \eqref{eq:rhosigma} hold for the operator $F$ in \eqref{eq:linear_e} for every $\sigma \geq 0$, 
and let \eqref{eq:Sobolev} and \eqref{eq:gamma} hold for $k_1=2$, $k_2=0$, define
\begin{equation*}
       \rho^a(\sigma) := \norm{\binom{T_{2,\sigma}^{-1}F_\sigma T_{2,\sigma}}{T_{1,\sigma}^{-1}\Pn{_\sigma}T_{1,\sigma}} - \binom{F}{\Pn}}.
\end{equation*}
Then we have the following estimates for the operator system $\binom{F}{\Pn}$ in \eqref{eq:linear_e},
\[\rho^a(\sigma)  \to 0 \text{ for } \sigma \to 0,\]
and in particular
\begin{equation*}
       \rho^a(\sigma) =\mathcal{O}(\rho(\sigma)+\gamma_1(\sigma)).
\end{equation*}
\end{lem}
Lemma \ref{le:properties} and \ref{le:compact} guarantee that Assumption \ref{assu:ori} is satisfied for the ambient operator equation \eqref{eq:linear_e}.
Assumptions \ref{assu:Surface} and \ref{ass:general} are verified as well because of Corollary \ref{coro:Surface_vector}, Lemma \ref{le:extended_operator}.
Then we can extend Theorem \ref{thm:directreg} and Theorem \ref{thm:Convergence} for \eqref{eq:linear_e}, the ambient operator equation for tangent vector fields, with exact and disturbed surfaces respectively.

\section{Examples}
In the following we present examples of applying Tikhonov regularization for an ill-posed problem and an image problem of which the solutions are vector fields defined on surfaces.

\label{sec:numerics}
\subsection{Magnetization reconstruction}
We consider a modelling for an inverse problem of reconstructing the Earth's magnetizations from measurement of the magnetic potential (see \cite{Ger16} for a recent reference), which consists in solving the operator equation
\begin{equation}
\label{eq:magnetization_3} 
 \vy = F \vu \text{ where } F \vu :=\frac{1}{4\pi} \int_{S_1^3} 
 \inner{\vu(x)}{\nabla_{y}\frac{1}{\abs{x-y}}}  ds(x).
\end{equation}
Here $\surface_1=S_1^3$ denotes the surface of the Earth, $\vu :S_1^3  \rightarrow \R^3$ denotes the vectorial magnetization 
of the Earth and $\vy: S_2^3\subset \R^3 \to \R$ denotes the magnetic potential 
data on the (satellite) orbit $\surface_2=S_2^3$. Moreover, $\inner{\cdot}{\cdot}$ denotes the Euclidean inner product in $\R^3$ 
and $\nabla_{y}$ denotes the gradient in Euclidean space with respect to $y$.  
We assume that the interior of the satellite orbit strictly contains $S_1^3$.

For the sake of simplicity of presentation we assume a $2D$-setting, that $\vu$, $\vy$ are 
constants in one Euler angle of $S_1^3$ and $S_2^3$ respectively. 
Then Equation \eqref{eq:magnetization_3} simplifies to 
\begin{equation}
\label{eq:magnetization} 
 \vy = F \vu \text{ where } F \vu :=\int_{S_1}  \inner{\vu(x)}{\nabla_{y} \log(\abs{y-x})} ds(x),
\end{equation}
where $S_1$ and $S_2$ denote the rectifiable, planar curves, which are the restrictions of $S_1^3$ and $S_2^3$ 
to the 2-dimensional Euclidean plane and $\inner{\cdot}{\cdot}$ denotes the Euclidean inner product in $\R^2$. 
For simplicity, we ignore a constant ($\frac{1}{2\pi}$) multiplication with the integral.
In the left image of Figure \ref{Fig:data}, the geometry of the experiment is sketched.
We plot simulated data $\vy$ according to some test data $\udag$, and some noisy data, $\vy^\delta$, which is 
obtained by adding Gaussian noise to $\vy$.
\begin{figure}[h]
\begin{center}
\includegraphics[width=0.52\textwidth]{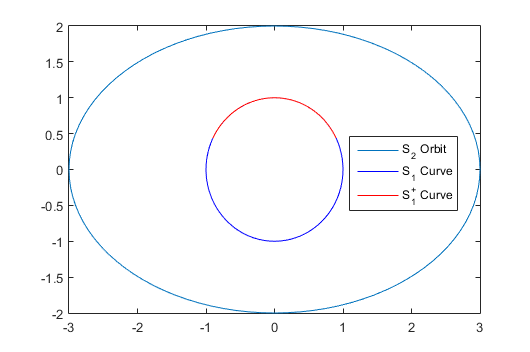}
\includegraphics[width=0.47\textwidth]{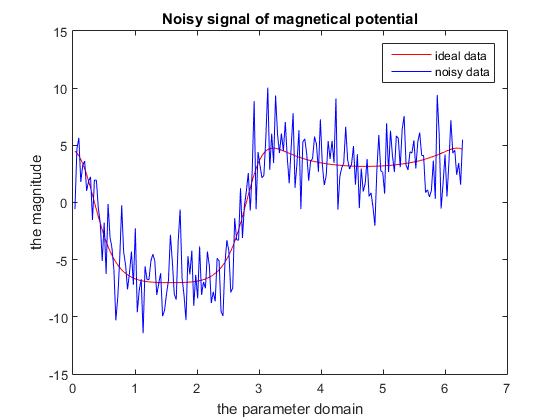}\\
\end{center}
\caption{The left image illustrates the problem setting. The right image shows some noisy magnetic potential data (with NSR=$0.5$) corresponding to the magnetization $\udag=[40x_1^3x_2, -40x_1^4]^T$.}
\label{Fig:data}
\end{figure}
As a case example, we assume that the lodestones are distributed only on a part of the upper half of the sphere $S_1^+$ 
of the earth. We denote functions in $\mathcal{H}^1(S_1)$, which have support on the upper hemi-sphere, by 
$\mathcal{H}^1(S_1^+)$.
This is numerically convenient, since this assumption allows to parametrize all functions with just one patch. 
We consider now $F$ as the operator
\[F:\mathcal{H}^1(S_1^+)\rightarrow L^2(S_2).\]
Let 
\begin{equation*}
\hat{C} := \sup_{y\in S_2} \set{\sup_{x\in S_1^+} \abs{\nabla_{y} (\log(\abs{y-x}))}},
\end{equation*}
which is finite, because $S_2$ and $S_1$ have a positive distance.
Then, because of 
\begin{equation*}
 \begin{aligned}
  \norm{F \vu}_{L^2(S_2)}^2 
  &= \int_{S_2} \left(\int_{S_1^+} \inner{\vu(x)}
                                         { \nabla_{y} (\log(\abs{y-x}))} ds(x) \right)^2\,ds(y)\\
  &\leq \hat{C}^2 \abs{S_2} \abs{S_1^+} \int_{S_1^+} \abs{\vu(x)}^2 \,ds(x)\\
 \end{aligned}
\end{equation*}
we have
\[\norm{F}\leq  \sqrt{\abs{S_2}\abs{S_1^+}} \hat{C}.\]

We point out that, in general, the 2-dimensional function $\vu$ in \eqref{eq:magnetization}
cannot be uniquely reconstructed from a 1-dimensional equation. 
If we restrict attention to tangential fields, that is $\vu^\dag$ is tangential to $S_1$, 
then the dimensions match, and one obtain a unique solution.
Here and in the later we will restrict to this case. 
Motivated by this, we consider regularization by the Tikhonov functional 
\begin{equation}
\label{eq:mag_tik}
\min_{\vu \in \hone{S_1^+}} \norm{F \vu -\vy^\delta }^2_{L^2(S_2)}+ \norm{\normal^T \vu}^2_{L^2(S_1)} + \alpha \abs{\vu}^2_{\mathcal{H}^1(S_1)}.
\end{equation}
On a one dimensional curve, using Definition \ref{def:spaces} the explicit form of the regularization functional becomes
\begin{equation*}
\abs{\vu}^2_{\mathcal{H}^1(S_1)}=\int_{S_1} (\partial_s\langle\vu, \normal \rangle) ^2+ (\partial_s \langle\vu,\tau \rangle)^2 ds(x),
\end{equation*}
where $\partial_s$ denotes the derivative along the curve $S_1$ with respect to arclength, and $\tau$ is then a unit tangent 
vector field of $S_1$.

\subsection*{Numerical tests}
In our numerical test example we assume that $S_1^+$ is the upper half part of a circle of radius $1$ and $S_2$ is an ellipse with short radius $2$ and long radius $3$.
Let $S_1^+$ be parametrized as the function graph $[x_1,x_2]^T=[t,\sqrt{1-t^2}]^T$, where $t\in [-0.9,0.9]$. We use the uniform 
cubic B-splines to approximate $S_1^+$ and note it as $S_1^+[h_1]$. $S_2$ is approximated piecewise linearly by polygons $S_2[h_2]$.
We first test an example of direct reconstruction without regularization, and the results (in Figure \ref{Fig:direct_inverse}) shows that the problem is highly ill-posed and the solution by a least square inversion completely losses the expected information.
\begin{figure}[h]
\begin{center}
\includegraphics[width=\textwidth]{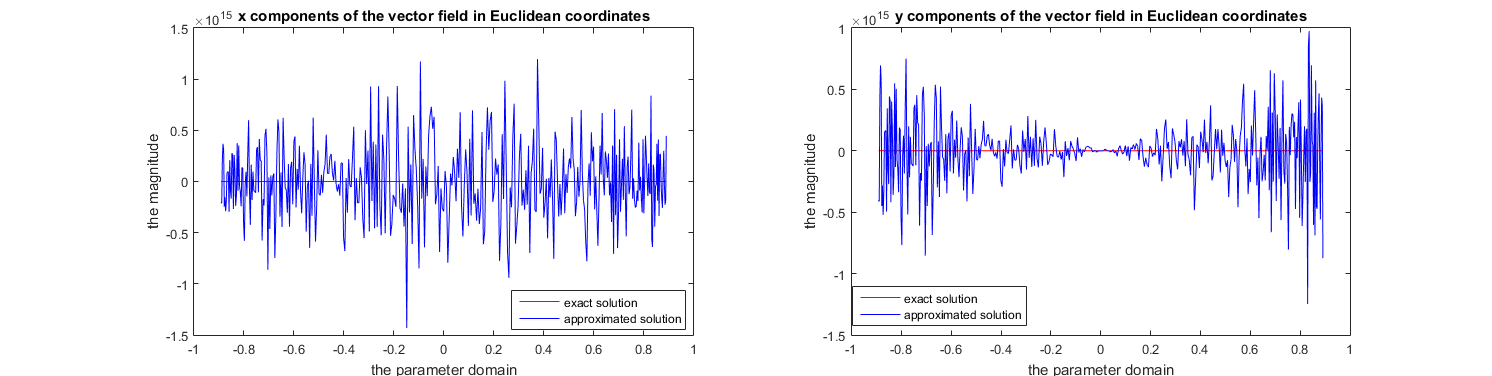}
\end{center}
\caption{A direct reconstruction without regularization.}
\label{Fig:direct_inverse}
\end{figure} 

Taking into account the discretization of the surfaces, we have the functional \eqref{eq:mag_tik} in an approximated form
\[ \norm{F_h \vu_h -\vy^\delta }^2_{L^2(S_2[h_2])}+\norm{\normal_h^T\vu_h}^2_{L^2(S_1[h_1])}  + 
\alpha \abs{\vu_h }^2_{\mathcal{H}^1(S_1[h_1])}. \]
The corresponding optimality condition is
\begin{equation}
\label{eq:optimal}
F_h^\star F_h\vu_h +\normal_h \normal_h^T\vu_h + \alpha\left(\tau_h \partial_{s_h}^2\langle\vu,\tau_h\rangle +\normal_h \partial_{s_h}^2 \langle\vu,\normal_h \rangle \right)=F_h^\star \vy^\delta,
\end{equation}
where $\tau_h$ and $\normal_h$ denotes the unit tangent and the unit normal vector field on $S_1^+[h_1]$ respectively, 
and $F_h^\star$ is the dual operator of $F_h$ in $L^2$ sense, that is 
\[F_h^\star\vy =\int_{S_2[h_2]} \vy(y) \frac{y-x}{\abs{y-x}^2} ds_h(y).\]

In the implementation, we use linear finite element methods for solving the problem \eqref{eq:optimal}. 
In the first example, we reconstruct the magnetization from the noisy data produced from a tangent vector field $\udag=[40x_1^3x_2, -40x_1^4]^T$ with $[x_1,x_2]^T\in S_1^+$ the coordinates in ambient space. The results are presented in Figure \ref{Fig:regularized_sol}, where we show a selection of plots by varying the noise level $\delta$ and the parameter $\alpha$, as well as the discretization scale $h$. The parameters are chosen to satisfy $\delta_k=C_1\alpha_k=C_2 h_k$ with $C_1$ and $C_2$ are constants, such that the assumption \eqref{eq:parameter} is fulfilled. 
The numerical results in Figure \ref{Fig:regularized_sol} are in accordance with the first statement of Theorem \ref{thm:Convergence}.
\begin{figure}
\begin{center}
\includegraphics[width=\textwidth]{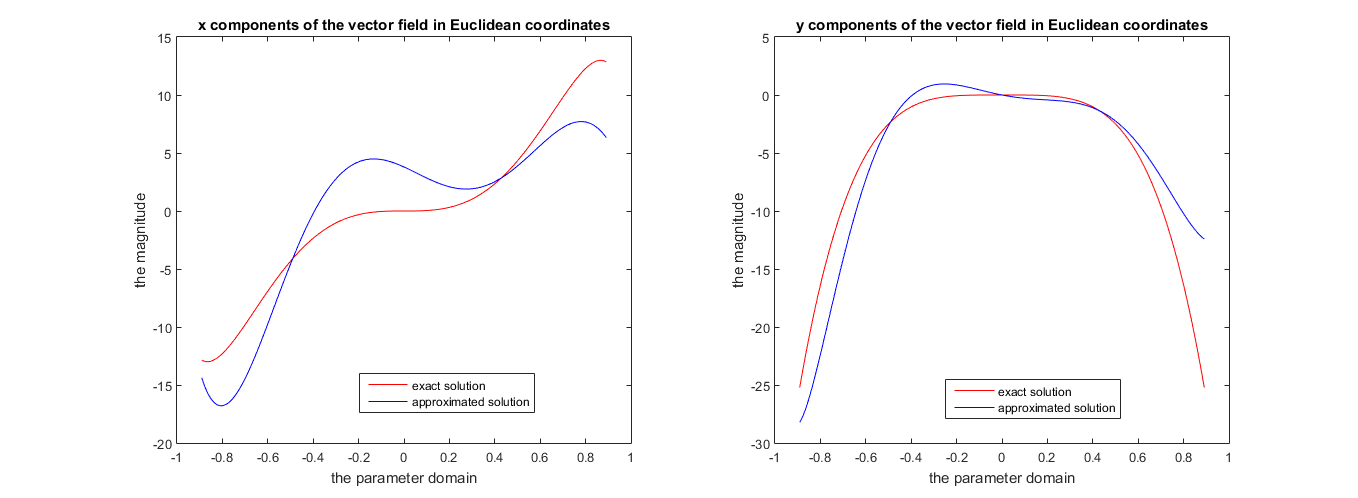}
\includegraphics[width=\textwidth]{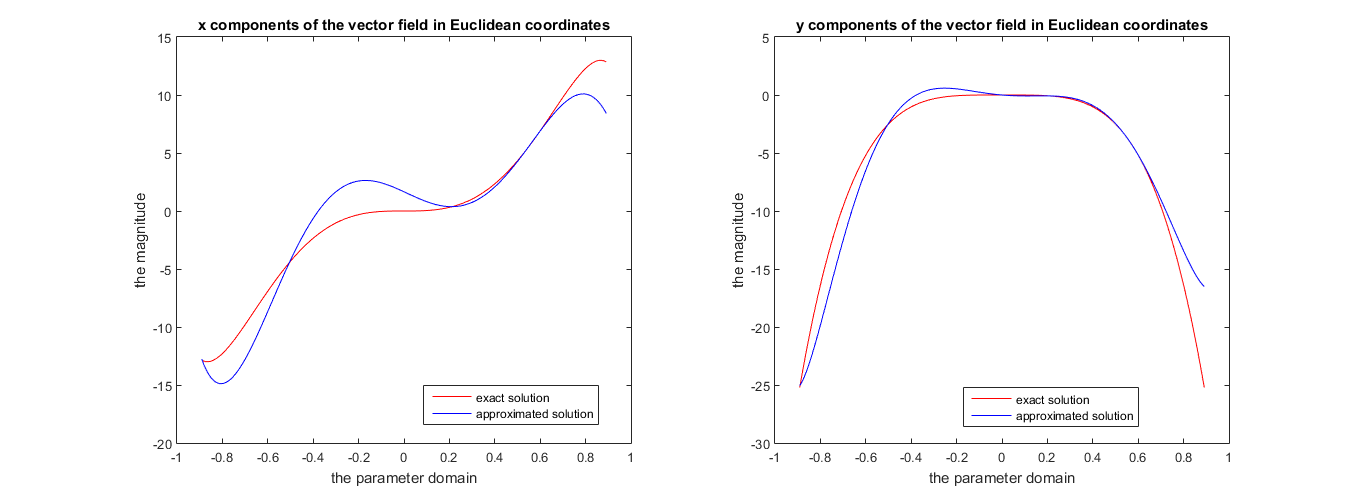}
\includegraphics[width=\textwidth]{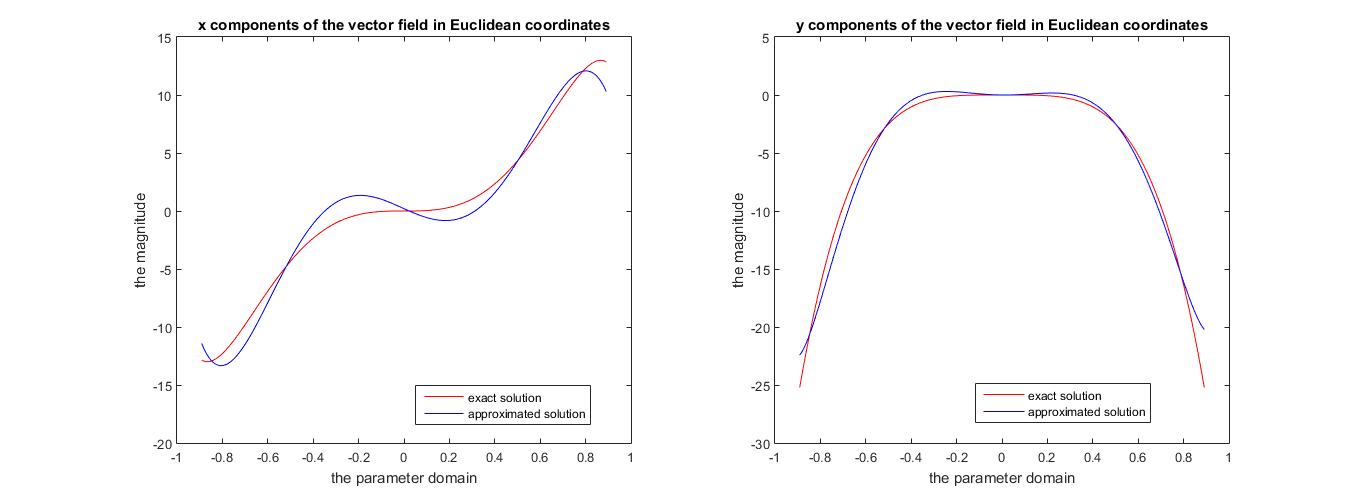}
\includegraphics[width=\textwidth]{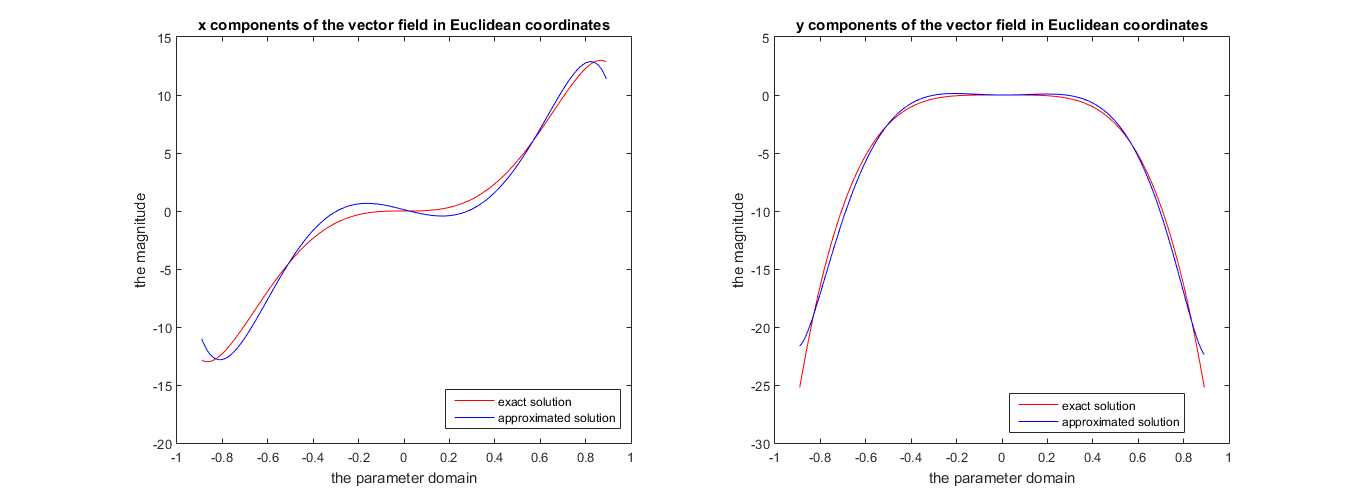}
\end{center}
\caption{The results obtained by minimizing the Tikhonov functional to approximate 
$\udag=[40x_1^3x_2, -40x_1^4]^T,\;[x_1,x_2]^T \in S_1^+$ with a decreasing level of noise, for decreasing regularization 
parameters and discretization sizes.}
\label{Fig:regularized_sol}
\end{figure} 

Another set of tests is made to distinguish the behaviour of the two semi-norms which are defined in \eqref{eq:h1} for regularization. In this example, we set the ideal solution to be $\udag(x)=[10x_2+5x_1 ,5x_2-10x_1]^T$ with $[x_1,x_2]^T \in S_1^+$, that is a vector field composed with constant amplitudes on both tangent and normal fields of $S_1^+$, and we let the noise signal ratio be $NSR=0.5$ in the data. Note that in this example, we do not enforce the tangential constraint, that is we are approximating the minimizer of
\[\norm{F \vu -\vy^\delta }^2_{L^2(S_2)}+ \alpha \abs{\vu}^2_{\hone{S_1}}.\]

The numerical results are shown in Figure \ref{Fig:constant}, where the results regularized by using the square of $\Hone{S_1}$ semi-norm \eqref{eq:sobolev} are also provided for comparison
\begin{equation}
\label{eq:sobolev}
\mathcal{R}(\vu)=\abs{\vu}^2_{\Hone{S_1}}.
\end{equation}
We visualize for the regularization parameter which gives the best approximation of the solution. 
We find that in this particular example, it is not possible to find a good reconstruction by minimizing with the 
functional \eqref{eq:sobolev}.
\begin{figure}[h]
\begin{center}
\includegraphics[width=\textwidth]{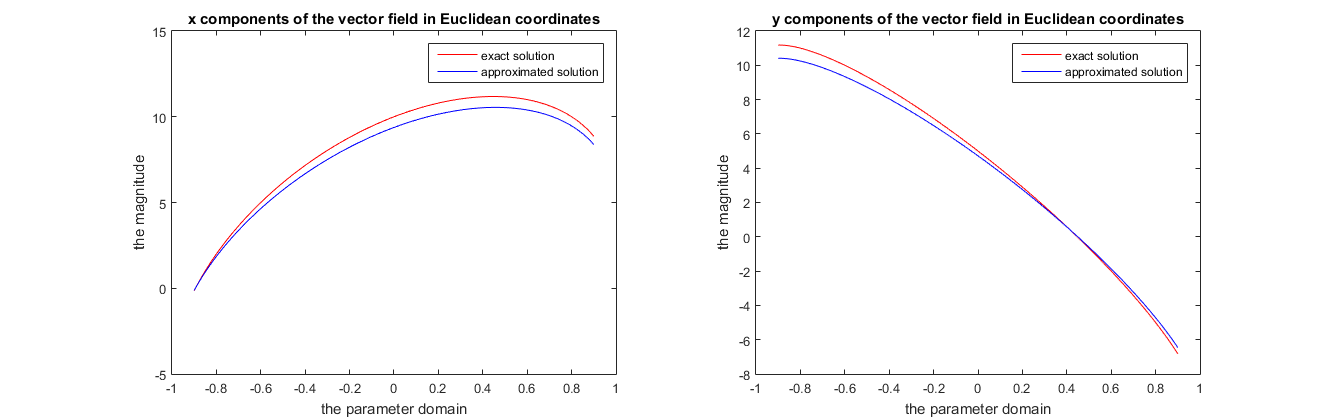}
\includegraphics[width=\textwidth]{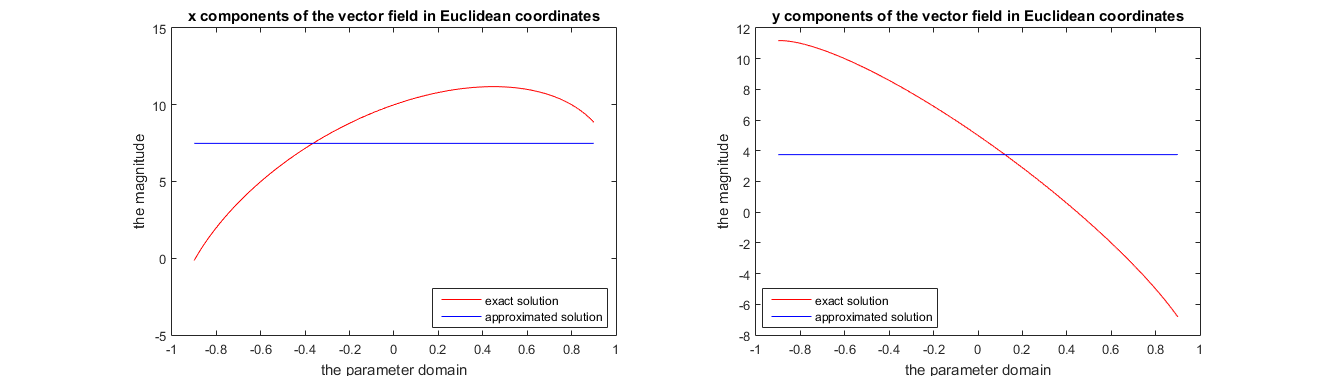}
\end{center}
\caption{The above two images plot the reconstruction with the squared $\hone{S_1}$-seminorm. The below ones are results with ordinary squared $\Hone{S_1}$-seminorm \eqref{eq:sobolev}. Here $\udag(x)=[10x_2+5x_1 ,5x_2-10x_1]^T=10\tau+5\normal.$}
\label{Fig:constant}
\end{figure} 

\subsection{Convergence rates for vector field denoising}
We consider denoising of a $2d$-vector-field $\vy^\delta = \vu^\delta$ defined on a $1$-dimensional curve $S$. In this example, we have the coincidence of the two surfaces, that is $\surface_1=\surface_2 =S$, and $F:\hone{S}\rightarrow L^2(S)$ is the embedding operator.

Assuming a sequence of approximating curves $(S_h)_{h \geq 0}$ of $S$ and data $(\vy_h^\delta)_{h \geq 0}$ defined on $S_h$, 
respectively,
Tikhonov regularization consists in minimizing the functional
\[\min_{\vu_h \in \mathcal{H}^1(S_h)} \norm{ \vu_h -\vu^\delta}^2_{L^2(S_h)}  + \alpha \abs{\vu_h }^2_{\mathcal{H}^1(S_h)}.\]

In our tests $S$ is the graph of a $sine$ function $[x_1,x_2]^T=[t,sin(t)]^T$ with $t\in [0,2\pi)$ 
(cf. Figure \ref{Fig:curves}). $S_h$ is approximated by uniform cubic B-splines.
In this way, $S_h$ is a $C^2$ approximation which satisfies Corollary \ref{coro:Surface_vector}.

We test for the synthetic solution 
\begin{equation*}
 \begin{aligned}
  \vu^\dag(x)
  =& 8(x_2) \tau+4\cos(x_1) \normal \\
  =& 4 \left[ \begin{array}{c} 
               \frac{2x_2\cos(x_1)-\cos(x_1)}{\sqrt{1+\cos^2(x_1)}}\\
               \frac{2x_2+\cos^2(x_1)}{\sqrt{1+\cos^2(x_1)}}
             \end{array} \right] \text{ for } x=[x_1,x_2]^T \in S.
 \end{aligned}
\end{equation*}

In this particular example the source condition \eqref{eq:sourcecondition} reads as follows:
\[\left(\tau\partial_{s}^2(8 x_2)) + \normal\partial_{s}^2(4\cos(x_1)) \right)  \in L^2(S),\]
which is easy to verify.

For the implementation, we approximate $\vu$ by piecewise linear functions with a uniform step size $h_{\vu}$. 
We denote the discretization size of $S_h$ by $h_{s}$. Then $S_{h}$ contributes a
surface disturbance to $S$ with an error bound of the order of $\gamma(h_s)$ according to our assumption. 
The results are shown in Table \ref{tab:Convergence_rates}, 
where we selectively show the residuals computed in the experiments.
We halve all the parameters for the selected steps, including the regularization parameter $\alpha_k$, the noise level $\delta_k$ and the discretization sizes $h_{s,k}$ and $h_{\vu,k}$. The parameters then satisfy $\frac{\delta_k^2}{\alpha_k}=\mathcal{O}(\delta_k)$, and the order of $\mathcal{O}(\gamma_k)$ is very close to the order of $\mathcal{O}(\alpha_k)$ in the tests, that is $\frac{\gamma_k^2}{\alpha_k}= \mathcal{O}(\gamma_k)$, $\gamma_k=\mathcal{O}(\delta_k)$.
We find that the numerical rates coincide with the results obtained in the second statement of Theorem \ref{thm:Convergence}
\[\abs{\check{\vu}_k-\vu^\dag}^2 = \mathcal{O}(\delta_k ).\]

\begin{figure}[h]
\begin{center}
\includegraphics[width=\textwidth]{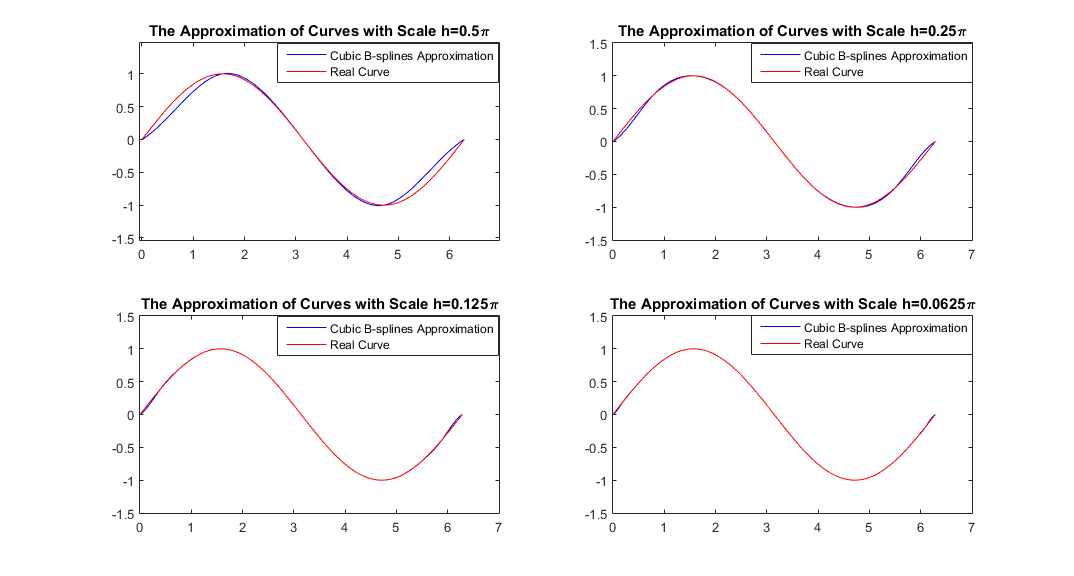}
\end{center}
\caption{Approximation of curves by different scales of discretization. }
\label{Fig:curves}
\end{figure} 

\begin{table}[h]
\setlength{\tabcolsep}{3pt}
\caption{Convergence rates of vector field denoising on curves}
\begin{center}
\begin{tabular}{ cccccc}
\hline
NSR($\frac{\norm{\delta}_{L^2}}{\norm{\vu^\dag}_{L^2}}$) & $1 $ & $0.5$  & $0.25$ & $0.125$ & $0.625$ \\
\hline
$\alpha$ &$ 0.04 $ & $ 0.02$  & $ 0.01$ &$ 0.005$ & $ 0.0025$   \\
\hline
&   $h_{s,1}=0.5\pi $  & $\gamma_1=1.8371$  & $h_{\vu,1}=0.02\pi $& \\
\hline
$\abs{\check{\vu}_1-\vu^\dag}^2$ & $\mathbf{366.3082}$  & $228.1245$ &$133.0704$  &$77.8783$ &$48.6980$ \\
\hline
&   $h_{s,2}=0.25\pi $  & $\gamma_2=0.8211$  & $h_{\vu,2}=0.01\pi $& \\
\hline
$\abs{\check{\vu}_2-\vu^\dag}^2$ & $347.8737$  & $\mathbf{200.5511}$ &$110.7511$  &$62.4273$ &$37.5464$ \\
\hline
&   $h_{s,3}=0.125\pi $  & $\gamma_3=0.3866 $   &$h_{\vu,3}=0.005\pi $& \\
\hline
$\abs{\check{\vu}_3-\vu^\dag}^2$& $276.7971$  & $166.7043$ &$\mathbf{94.3003}$  &$54.0387$ &$33.1077$ \\
\hline
&   $h_{s,4}=0.0625\pi $  & $\gamma_4=0.1922 $   & $h_{\vu,4}=0.0025\pi $&\\
\hline
$\abs{\check{\vu}_4-\vu^\dag}^2$ & $242.3850$  & $150.7489$ &$90.4440$  &$\mathbf{55.0508}$ &$34.8112$ \\
\hline
&   $h_{s,5}=0.03125\pi $  & $\gamma_5=0.0971 $   &$h_{\vu,5}=0.00125\pi $ &\\
\hline
$\abs{\check{\vu}_5-\vu^\dag}^2$ & $268.2314$  & $158.8666$ &$90.0663$  &$52.4830$ &$\mathbf{32.7122}$ \\
\hline
\end{tabular}
\end{center}
\label{tab:Convergence_rates}
\end{table}

\section{Conclusion}
In this paper we have studied Tikhonov regularization for solving ill--posed operator equations where the 
solutions are functions defined on surfaces, 
and especially we emphasize on the functions with range in vector bundles.
Such problems appear in a variety of applications
such as recovering magnetization from magnetic potential, vector fields denoising on surfaces and so on.
We extended the existing theory on approximation of infinite dimensional Tikhonov regularized solutions 
for ill--posed operator equations to the surface setting.
The theory has been generalized to the case of vector fields,
where they are represented by vector valued functions associated to the coordinates in ambient spaces of the surfaces. 
The additional features of this theory are that it allows to take into account perturbations of the surface and the vector bundle,
and it provides an analysis on the convergence and convergence rates which may give an optimal regularization parameter choice. 

\section*{Acknowledgement}
This work has been supported by the Austrian Science Fund (FWF)
within the national research network Geometry and Simulation, projects S11704, S11707. 
Guozhi Dong thanks Christian Gerhards for leading him to the application on magnetization problems 
and the productive discussions. All the authors thank the reviewers and the editor for their comments 
to help improving the paper.

\appendix
\section{Background on surfaces and vector fields}
\label{appendix:background}
Throughout this paper we use some geometrical notations which are collected in Table \ref{tab:Geo_Notation}. 
\begin{table}[h]
\setlength{\tabcolsep}{3pt}
\caption{Notation corresponding to the geometry}
\begin{center}
\begin{tabular}{ cccccc}
\hline
Notation & Remark  & Notation & Remark   \\
\hline
$\surface$ & a parametrizable surface & $g$ & metric tensor   \\
\hline
$\parametrization$ & mapping $\parametrization:\Omega\rightarrow \surface$ &$\partial \parametrization$ & Jacobian of $\parametrization$ \\
\hline
$\mathcal{TM}$ & tangent vector bundle& $\mathcal{NM}$ & normal vector bundle\\
\hline
$\Ptm$ & tangent projection & $\Pn$ & normal projection \\
\hline
$\normal$ & unit normal vector & $g(\cdot,\cdot)$& inner product on $\mathcal{TM}$  \\
\hline
$\partial_i \parametrization$ & tangent basis vector &   \\
\hline
\end{tabular}
\end{center}
\label{tab:Geo_Notation}
\end{table}

Let $(\surface,g)$ be a $d$-dimensional smooth, oriented, connected and compact surface in $\R^{d+1}$, which has bounded curvature, and 
is associated with the Riemannian metric $g$. To simplify the representation we assume that the surface $\surface$ 
is parametrizable, that is 
       $\surface \subseteq \R^{d+1}$ can be represented by a map $\parametrization: \Omega \to \R^{d+1}$,
       \begin{equation*}
          \surface = \parametrization (\Omega),
       \end{equation*}
where $\Omega\subset \R^d$ and $\parametrization:\Omega\rightarrow \surface$ is a regular parametrization. In particular we assume that $\parametrization$ is bijective and the inverse $\parametrization^{-1}:\surface\rightarrow \Omega$ is also a regular map.
$\mathcal{T}_x\surface$, $\mathcal{N}_x\surface$ denote the tangent and normal vector spaces at $x \in \surface$,  respectively.
Then the tangent and normal bundle are defined by 
       \begin{equation*}
        \mathcal{TM} = \bigcup_{x \in \surface,\;} \set{(x,\vv): \vv \in \mathcal{T}_x\surface}
       \end{equation*}
       and 
       \begin{equation*}
        \mathcal{NM} = \bigcup_{x \in \surface,\;} \set{(x,\vv): \vv \in \mathcal{N}_x\surface}\,,
       \end{equation*}
respectively.
We denote by 
       \begin{equation*}
       \partial \parametrization (\zeta) = [\partial_1 \parametrization,\partial_2 \parametrization,\cdots,\partial_d \parametrization](\zeta) \in \R^{(d+1) \times d}
       \end{equation*}
the \emph{Jacobian} of the parametrization at $x = \parametrization(\zeta) \in \surface$. 
The parameters in $\Omega$ are denoted by $\zeta$. The derivatives with respect to these parameters are always denoted by $\partial$.
The \emph{metric tensor} $g$ is related to the parametrization $\parametrization$ by 
       \begin{equation}
        \label{eq:metric}
        g(\zeta) =(\partial \parametrization(\zeta))^T \partial \parametrization(\zeta) \in \R^{d \times d}
       \end{equation}
       on $\surface$ at $x=\parametrization(\zeta)$. 
We call a vector field a \emph{tangent vector field} if $\tilde{\vv}:\surface\rightarrow \R^{d+1}$ with range in the tangent bundle $\mathcal{TM}$, that is, $\tilde{\vv}(x) \in \mathcal{T}_x\surface$ for all $x \in \surface$.
       Every tangent vector $\tilde{\vv}(x) \in \mathcal{T}_x \surface$ can be represented in terms of the tangential basis 
       $(\partial_i \parametrization(\zeta))_{i=1,2,\cdots,d}$, at each $x = \parametrization(\zeta)$, via 
       \begin{equation}
       \label{eq:para}
        \tilde{\vv}(x) = \partial \parametrization (\zeta) (\hat{v}_1(\zeta),\cdots,\hat{v}_{d}(\zeta))^T \in \R^{(d+1) \times 1}.
       \end{equation}
 For two tangent vector field $\tilde{\vv}(x)$ and $\tilde{\vu}(x)$, using \eqref{eq:para}, we have the relation
        \begin{eqnarray*}
        g(\tilde{\vv},\tilde{\vu})&:=&(\hat{v}_1(\zeta),\cdots,\hat{v}_{d}(\zeta))\underbrace{\left(\partial \parametrization (\zeta)\right)^T\partial \parametrization (\zeta)}_{=g(\zeta)}(\hat{u}_1(\zeta),\cdots,\hat{u}_{d}(\zeta))^T \nonumber\\
        &=&\tilde{\vv}^T \tilde{\vu}.        
        \end{eqnarray*}       
We denote by $\normal(x)= (n_1(x),n_2(x),\cdots,n_{d+1}(x))^T\in \R^{d+1}$ the unit normal vector of 
      $\surface$ at $x\in \surface$ with fixed orientation, and by $\Ptm$ and $\Pn$ the projections onto 
      $\mathcal{T}\surface$, $\mathcal{N}\surface$ at a point $x \in \surface$, respectively. They are represented by matrices 
      \begin{equation}
      \label{eq:pn}
         \Pn(x) = \normal(x) \normal(x)^T \in \R^{(d+1) \times (d+1)} \quad \text{and} \quad  \Ptm(x) = \mathcal{I} - \Pn(x) \in \R^{(d+1) \times (d+1)},
      \end{equation} 
where $\mathcal{I}$ is the $(d+1) \times (d+1)$ identity matrix. Because $\Ptm$ and $\Pn$ are orthogonal projections on the surface they are satisfying the pointwise estimate
\begin{equation*}
\max \set{\abs{\Ptm \vu},\abs{\Pn \vu}} \leq \abs{\vu}\;.
\end{equation*}
In Figure \ref{fig:vectorfields} we give a schematic representation of the surface $\surface$, its parametrization $\parametrization$ over the coordinate domain $\Omega$, as well as the vector spaces $\mathcal{T}_x\surface$ and $\mathcal{N}_x\surface$. 
\begin{figure}[ht]
\centering
\begin{picture}(170,70)
\put(65,25){$\parametrization$}
\put(130,25){$\surface$}
\put(135,70){$\mathcal{N}_x\surface$}
\put(170,45){$\mathcal{T}_x\surface$}
\put(130,36){$x$}
\put(30,15){$\zeta$}
\put(20,5){$\Omega$}
\includegraphics[width=0.5\textwidth]{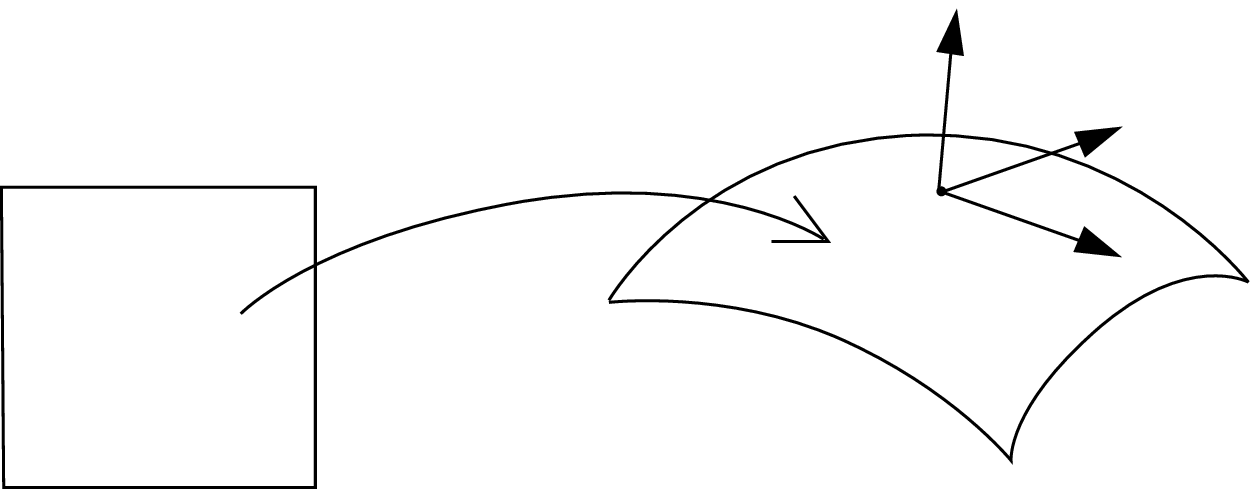}
\end{picture}
\caption{Geometry mapping and vector spaces}
\label{fig:vectorfields}
\end{figure}
 
\begin{lem} 
Let $A \in \R^{(d+1) \times (d+1)}$, then for every $x \in \surface$ 
\begin{equation}
\label{eq:AA}
 \abs{\Ptm(x) A} \leq \abs{A}\,,\text{ and } \abs{\Pn(x) A} \leq \abs{A}
\end{equation}
where $\abs{A}$ denotes the Frobenius norm.
\end{lem}
\begin{proof}
For every matrix $A$, $\abs{A}^2=\trace (A^T A)$, and because $\Ptm (x)$ and $\Pn (x)$ are orthogonal projections, 
we have
\[
\begin{aligned}
 \abs{A}^2 &= \abs{\Ptm(x) A +\Pn(x) A}^2 \\
 &= \trace \left((\Ptm(x) A +\Pn(x) A)^T(\Ptm(x) A +\Pn(x) A)\right)\\
 &= \trace \left((\Ptm(x) A)^T \Ptm(x) A + (\Pn(x) A)^T(\Pn(x) A)\right)\\ 
 &= \trace \left((\Ptm(x) A)^T \Ptm(x) A\right) + \trace \left((\Pn(x) A)^T(\Pn(x) A)\right)\\
 &= \abs{\Ptm(x) A}^2 +\abs{\Pn(x) A}^2.
\end{aligned}
\]
\end{proof}

Moreover, we have the relation between the parametrization and the projection
      \begin{equation*}
       \Ptm(x) = \partial m(\zeta)g^{-1}(\zeta) (\partial m(\zeta))^T 
      \end{equation*}
      for $x = m(\zeta)$. We denote by 
      \begin{equation}\label{eq:projectiona}
        (\partial \parametrization (\zeta))^\dagger = g^{-1}(\zeta) (\partial \parametrization)^T(\zeta) \in \R^{d \times (d+1)},
      \end{equation}
      the Moore-Penrose inverse of the matrix $\partial \parametrization (\zeta) \in \R^{(d+1) \times d}$, which satisfies
      \begin{equation*}
       (\partial \parametrization (\zeta))^\dagger \partial \parametrization (\zeta) = \id(\zeta) \in \R^{d \times d} \text{ and }
       \partial \parametrization (\zeta) (\partial \parametrization (\zeta))^\dagger = \Ptm (m(\zeta)) \in \R^{(d+1) \times (d+1)}.
      \end{equation*}
      
      \begin{defi}
      \label{de:gradient}
       For a scalar field $v : \surface \to \R$, we define its \emph{surface gradient} at a point $x=\parametrization(\zeta)$ by
       \begin{equation} \label{eq:nabla_m}
          \gradientM v(x):= \left(\partial (v \circ \parametrization)\; (\partial \parametrization)^\dagger \right)\circ m^{-1}(x)
          \in \R^{1 \times (d+1)}.
       \end{equation}
       The surface gradient fulfills the relation 
       \begin{equation}\label{eq:relations_gradients}
       \gradientM v(\parametrization(\zeta))\; \partial \parametrization(\zeta) = 
       \partial (v \circ \parametrization)(\zeta) \in \R^{1 \times d}.
       \end{equation}
       Given a vector valued function $\vv:\surface\rightarrow \R^{d+1}$, the definition of the 
       \emph{gradient with respect to the surface}  
       $\gradientM$ consists in taking the gradient of each scalar component $v_i$, i.e., for $x = \parametrization(\zeta)$ we have
       \begin{equation}
       \label{eq:vnabla_m}
       \gradientM \vv(x):=
        \left(\begin{array}{c}
	\gradientM v_1(x) \\
	\cdots \\
	\gradientM v_{d+1}(x) 
        \end{array}\right)=
        \left(\partial (\vv \circ \parametrization)(\zeta) (\partial \parametrization (\zeta))^\dagger\right)\circ m^{-1}(x)
        \in \R^{(d+1) \times (d+1)}.
       \end{equation}
       \end{defi}
       Every vector field $\vv:\surface\rightarrow \R^{d+1}$ can be decomposed into the two vector fields 
       $\Ptm \vv$ and $\Pn \vv$ which have ranges in 
       $\mathcal{T}_x \surface$ and $\mathcal{N}_x\surface$ respectively for all $x \in \surface$.   
       
       In the following we recall the definition of the \emph{covariant derivative} of a \emph{tangent vector field} 
       $\tilde{\vv} : \surface \to \TB$ (see \cite{Lee97,Car92}). 
       Equation \eqref{eq:para} states that 
       \begin{equation*}
       \tilde{\vv}(\parametrization(\zeta)) = \partial m(\zeta) (\hat{v}_1(\zeta),\cdots,\hat{v}_d(\zeta))^T = 
       \sum_{k=1}^d \hat{v}_k(\zeta) \partial_k \parametrization(\zeta),
       \end{equation*}
        where $\hat{v}_k : \Omega \to \R$ denotes the local coordinate representation with respect to the basis 
       $(\partial_k \parametrization)_{k=1,2,\cdots,d}$.
       
       \begin{defi}[refering to \cite{Car92,DubFomNov91}]\label{de:covariant}
       The \emph{covariant derivative} in direction $k$ with respect to tangent basis $\partial m$, 
       $(\widetilde{\nabla}_k \tilde{\vv})_{k=1,2,\cdots,d}$, is defined by the orthogonal projection of 
       \begin{equation*}
        \partial_k (\tilde{\vv} \circ \parametrization)(\zeta)
        =  \sum_{i=1}^d \partial_k \hat{v}_i(\zeta) \partial_i \parametrization(\zeta) + 
                                      \sum_{i=1}^d \hat{v}_i(\zeta) \partial_{ik} \parametrization(\zeta)           
       \end{equation*}
into the tangent space. That is, for $x=m(\zeta)$, the \emph{covariant derivative} is given by  
       \begin{equation}\label{eq:covariant}
          \widetilde{\nabla} \tilde{\vv} (\parametrization(\zeta)) = \Ptm(m(\zeta))
          [ \partial_1 (\tilde{\vv} \circ \parametrization)(\zeta),
          \cdots,
            \partial_d (\tilde{\vv} \circ \parametrization)(\zeta)].
       \end{equation}
       \end{defi}
In particular for a tangent vector field $\tilde{\vv}$ at $x = m(\zeta)$ it follows from \eqref{eq:covariant} and 
       \eqref{eq:vnabla_m} that
       \begin{equation}
       \label{eq:cov_m_t}
         \widetilde{\nabla} \tilde{\vv} (x) (\partial \parametrization (\zeta))^\dagger = \Ptm(x) \gradientM \tilde{\vv} (m(\zeta)).
       \end{equation}
Moreover, for an arbitrary vector field $\vv :\surface \to\R^{d+1}$ we have
       \begin{equation*}
        \widetilde{\nabla} (\Ptm \vv) (x) (\partial \parametrization (\zeta))^\dagger = \Ptm(x) \gradientM (\Ptm \vv) (x).
       \end{equation*}

To conclude the appendix, we present a few auxiliary results, which are used to characterize spaces in Section \ref{sec:vectorfields}.
 \begin{lem}\label{le:invariant}
The gradient operator $\gradientM$ is independent of the parameterization,  and fulfils standard rules of differentiation, such as the product rule
\begin{enumerate}
 \item Let $v,w:\surface \to \R$ be differentiable, then 
       \begin{equation*}\gradientM (vw)=v\gradientM w+ w\gradientM v.\end{equation*}
 \item Moreover, let $\vv,\vw,\vz: \surface \to\R^{d+1}$, then  
       \begin{equation*}
        \begin{aligned}
          \gradientM (\vv^T \vw) &= \vv^T\gradientM \vw+ \vw^T\gradientM \vv,\\
          \gradientM (v \vw) &= \vw\gradientM v+v \gradientM \vw,\\
          \gradientM ((\vv^T \vw)\vz) &= \vz \vv^T\gradientM \vw+ \vz \vw^T\gradientM \vv +(\vv^T \vw) \gradientM \vz.
        \end{aligned}
       \end{equation*}
\end{enumerate}
 \end{lem}

\begin{lem}
\label{le:Pro_gradient}
Let $\normal(x)=(n_1(x),\cdots,n_{d+1}(x))^T$ the unit normal vector field on $\surface$, and let $\vv:\surface \rightarrow \R^{d+1}$ be a differentiable vector field, then
\begin{enumerate}
\item The following formulas hold
		\begin{equation*}
        \normal^T\gradientM \normal=0, \quad \Ptm \gradientM \normal=\gradientM \normal \text{  and  } \Ptm \normal = 0\;,
		\end{equation*} 
		and    
		\begin{equation*}
		 \begin{aligned}
  		\Ptm \gradientM(\Ptm \vv) &= \Ptm \gradientM \vv-(\normal^T \vv)\gradientM \normal,\\
  		\Pn\gradientM(\Pn \vv) &= \Pn\gradientM \vv +\normal \vv^T\gradientM \normal.
		 \end{aligned}
		\end{equation*}
\item  Moreover, the seminorm can be represented by 
	\begin{equation*}
	 \abs{\vv}^2_{\hone{\surface}}=\norm{\gradientM \vv -(\normal^T \vv)\gradientM \normal +\normal \vv^T\gradientM \normal}^2_{\leb{2}}.
	\end{equation*}
\end{enumerate}
\end{lem}
\begin{proof}
For the first item, since $\normal$ is a unit normal vector field, 
\begin{equation*}\normal^T(x) \normal(x) =1, \quad \forall x\in \surface.\end{equation*}
Then from Lemma \ref{le:invariant} it follows that
\begin{equation*}2\normal^T(x)\gradientM \normal(x)=\gradientM (\normal^T(x) \normal(x))=0.\end{equation*}
Moreover, $\Ptm(x)=\mathcal{I}-\normal(x)\normal^T(x),$ and thus from the previous it follows that
\begin{equation*}\Ptm(x)\gradientM \normal(x)=\mathcal{I}\gradientM \normal(x)-\normal(x)\normal^T(x)\gradientM \normal(x)=\gradientM \normal(x).\end{equation*}
We apply the product rules in Lemma \ref{le:invariant} and properties of $\normal$ derived above, which show that
\begin{equation*}
\begin{aligned}
\Ptm\gradientM(\Ptm \vv) &= \Ptm \gradientM( \vv-(\normal^T \vv)\normal)\\
&= \Ptm ( \gradientM \vv-\normal\normal^T\gradientM \vv- \normal\vv^T\gradientM \normal- (\normal^T \vv)\gradientM \normal)\\
&= \Ptm \gradientM \vv - (\normal^T \vv)\gradientM \normal,
\end{aligned}
\end{equation*} 
and
\begin{equation*}
\begin{aligned}
\Pn\gradientM(\Pn \vv) &= \Pn\gradientM((\normal^T \vv)\normal)\\
& = \Pn (\normal\normal^T\gradientM \vv+ \normal\vv^T\gradientM \normal+ (\normal^T \vv)\gradientM \normal)\\
&= \Pn \gradientM \vv +\normal\vv^T\gradientM \normal.
\end{aligned}
\end{equation*}
\item Adding the two identities, and using the orthogonality of $\Ptm$ and $\Pn$ on $\leb{2}$, we get the second item from Definition \ref{def:spaces}:
\begin{equation*}
\begin{aligned}
~ & \norm{\gradientM \vv -(\normal^T \vv)\gradientM \normal +\normal \vv^T\gradientM \normal}^2_{\leb{2}}\\
=& \norm{\Ptm \gradientM(\Ptm \vv)+\Pn\gradientM(\Pn \vv)}^2_{\leb{2}}\\
=& \norm{\Ptm \gradientM(\Ptm \vv)}^2_{\leb{2}}+\norm{\Pn\gradientM(\Pn \vv)}^2_{\leb{2}}\\
=&\abs{\vv}^2_{\hone{\surface}}.
\end{aligned}
\end{equation*}
This concludes the proof.
\end{proof}

\end{document}